\RequirePackage{fix-cm}
\documentclass[smallextended]{svjour3}       
\smartqed  

\usepackage{amssymb, amsmath, amsfonts, graphicx,latexsym}
\usepackage{fullpage}
\usepackage{multirow}
\usepackage{caption, subcaption}
\usepackage{changes}
\usepackage{placeins}
\usepackage{authblk}
\usepackage[square,sort,comma,numbers]{natbib}
\usepackage{xcolor}
\usepackage{hyperref}

\begin{document}

\title{Symmetric relative equilibria with one dominant and four infinitesimal point vortices}

\titlerunning{Symmetry in the $(1+4)$-vortex problem}        

\author{Alanna Hoyer-Leitzel     \and
        Sophie Phuong Le 
}


\institute{Alanna Hoyer-Leitzel \at
              Department of Mathematics and Statistics, Mount Holyoke College, 50 College Street, 
              South Hadley, MA 01075 \\
              Tel.: +1-716-499-5988\\
              \email{ahoyerle@mtholyoke.edu}           
           \and
           Sophie Phuong Le \at
           		Department of Mathematics and Statistics, Mount Holyoke College             
}

\date{Received: date / Accepted: date}

\maketitle

\begin{abstract}
We investigate the symmetry of point vortices with one dominant vortex and four vortices with infinitesimal circulations in the $(1+4)$-vortex problem, a subcase of the five-vortex problem. The four infinitesimal vortices inscribe quadrilaterals in the unit circle with the dominant vortex at the origin. We consider symmetric configurations which have one degree of spacial freedom, namely the $(1+N)$-gon, kites, rectangles, and trapezoids with three equal sides. We show there is only one possible rectangular configuration (up to rotation and ordering of the vortices) and one possible trapezoid with three equal sides (up to rotation and ordering), while there are parametrically defined families of kites.  Additionally we consider the $(1+4)$-gon and show that the infinitesimal vortices must have equal circulations on opposite corners of the square. The proofs are heavily dependent on techniques from algebraic geometry and require the use of a computer to calculate Gr\"{o}bner bases.
\end{abstract}

\section{Introduction}
The Hamiltonian $n$-vortex problem is a historical problem, first posed by Kirchhoff \cite{kirchhoff1883vorlesungen}. However, the study of configurations of the vortices has modern applications, and we can find examples of vortex configurations similar to those given by relative equilibria in the $n$-vortex problem, both in nature and experimentally. For example, Hurricane Isabel in 2003 was documented to have six distinct mesovortices making up the eye of the hurricane: one in the center and five symmetrically located around the center \cite{kossin2004mesovortices}, much like the $(1+5)$-gon configuration in the $n$-vortex problem. Sheets of water created in low gravitational environments form vortices that also mimic rhombus and pentagon shaped relative equilibria \cite{nasa2012youtube}. In another experiment of fluids in a rotating tank designed to mimic vortices on giant plants, we see one large vortex surrounded by smaller vortices created by sheer instability, creating vortex configurations like those given by the $(1+N)$-vortex problem \cite{uclaspin2012youtube}. This paper examines possible symmetric vortex relative equilibria with one larger vortex surrounded by $N=4$ smaller vortices.

Relative equilibria of Hamiltonian $n$-point problems are periodic self-similar solutions  that rotate around their center of mass, vorticity, or the appropriate equivalent quantity. In a rotating coordinate system, these solutions correspond to an $n$-torus of degenerate fixed points in the phase space. Modern conceptions of the problem of relative equilibria with a dominant mass were first formulated by Hall \cite{hall1988central} and Moeckel \cite{moeckel1994linear}. Of particular note is the definition of relative equilibria of the $(1+N)$-body problem with one large and $N$ infinitesimal masses as a limit of relative equilibria with one large and $N$ small but positive masses. 

In this same vein, Barry et al \cite{barry2012relative} and Barry and Hoyer-Leitzel \cite{barry2016existence}, define relative equilibria of the $(1+N)$-vortex problem as the limiting case of relative equilibria with one large and $N$ small vortices of any circulation, as the small vortices become infinitesimal. These papers derive a sort of potential function $V(\theta)$ whose critical points correspond to positions of vortices in relative equilibria of the limiting problem, and whose Hessian gives the linear stability of relative equilibria when continued back from the limit. These results are summarized in Theorems \ref{thm:V} and \ref{thm:existence} in this paper. These theorems assume that in the limit, vortices are bounded away from each other. Loosening these assumptions would make for interesting future work.

The focus of this paper is to classify the symmetric configurations of relative equilibria of the $(1+N)$-vortex problem when $N=4$. This corresponds to an examination of the symmetry of critical points of $V(\theta)$. In the limit, the large vortex moves to the origin, and the infinitesimal vortices move to the unit circle. The symmetry of the configurations is dependent on the symmetry of points positioned around the unit circle, inscribing a convex quadrilateral within the unit circle. We consider symmetric quadrilaterals that have only one degree of freedom, after accounting for rotational symmetry around the unit circle. These quadrilaterals are squares (a $4$-gon), kites, rectangles, and trapezoids with three equal sides. Section \ref{sec:symmetry defs} defines configurations in detail. 

The $(1+4)$-vortex problem is a subcase of the five-vortex problem, though the categorization of the positions of four bodies or four vortices in relative equilibria is well developed and the types of symmetry are related to those used in this paper. In the $n$-body problem, relative equilibria fall within a larger group of configurations called central configurations, and for four bodies, these are classified as concave or convex, with the set of convex configurations further classified by different symmetric or asymmetric quadrilaterals \cite{albouy1996symmetric,albouy2008symmetry,corbera2018four,corbera2019classifying,deng2017four,rusu2016bifurcations,corbera2014central,santoprete2021uniqueness,santoprete2021uniqueness2,sun2023uniqueness}. 

Certain cases share the exact definitions or similar results to those in this paper. In \cite{cors2012four}, Cors and Roberts classify co-circular symmetric central configurations where all four bodies lie on the same circle. They find the only symmetric co-circular configurations are kites and isosceles trapezoids (a general case containing rectangles and trapezoids with three equal sides). Removing the assumption of co-circular, both Long and Sun in \cite{long2002four} and Perez-Chavela and Santoprete in \cite{perez2007convex} find that there are symmetric configurations with two equal pairs of masses on opposite vertices of a rhombus, similar to the result in Theorem \ref{thm:square} in this paper. In comparison, relative equilibria in the four vortex problem with two equal pairs of vortices are throughly classified in \cite{hampton2014relative} by Hampton, Roberts, and Santoprete. Again the positions of equal pairs in rhombus and kite configurations are similar to results in Theorem \ref{thm:square} and Theorem \ref{thm:kite}. 

Results for five-point problems are equally interesting. 
In \cite{hampton2009finiteness}, Hampton shows finiteness of the number of kite configurations in the five-body and and five-vortex problems, and in \cite{roberts1999continuum}, Roberts finds a continuum of central configurations in the five-body problem that continue across potential functions to the five-vortex problem. In \cite{lee2009central} Lee and Santoprete calculate all possible planar central configurations for five equal masses. For the five-vortex problem, Oliveira and Vidal \cite{oliveira2019stability} calculate the linear stability of the rhombus with the central vortex relative equilibrium, where the rhombus is made of two pairs of equal vortices at opposite corners. In contrast, Marchesin and Vidal examine the restricted five-vortex problem with two equal pairs of vortices in a rhombus and one infinitesimal vortex \cite{marchesin2019global}. This case is different than the one presented in this paper in that no limit is needed in defining the relative equilibria in the restricted problem.

Because, in the limiting case considered in this paper, the infinitesimal vortices lie on the unit circle with the dominant vortex at the origin, the $(1+N)$-vortex problem is a subset of vortex ring problems where vortices are arranged in a regular $N$-sided polygon with or without a vortex in the center.  
Cabral and Schmidt \cite{cabral2000stability} define and look at the $(N+1)$-vortex ring with $N$ equal vortices with circulation $\Gamma=1$ in a polygon and one central vortex with circulation $\kappa$. They find that the $(4+1)$-gon is stable when the central vortex has circulation $\kappa \in (-\frac12,\frac94)$. Ohsawa gives a sufficient condition on nonlinear stability of relative equilibria of the $N$-vortex problem and considers examples of the $(N+1)$-gon for $N=3$ and $N=4$ \cite{ohsawa2024nonlinear}. There is also a paper by Newton and Chamoun \cite{newton2009vortex} on vortex lattice theory where examples include the $(4+1)$-vortex ring with four equal vortices on the corners of the square.
On the other hand, Barry et al \cite{barry2012relative} consider the $(1+N)$-gon with $N$ infinitesimal vortices with the same circulation, which is linearly unstable regardless of the sign of the small vortices when $N\geq 4$. In Theorems \ref{thm:square} and \ref{square stability}, we prove the $(1+4)$-gon must have two equal pairs of vortices on opposite corners, and is always linearly unstable. 

There are rich results on the symmetry and number of relative equilibria in $(1+N)$-body and vortex problems with $N=3$ and $N=4$. Corbera et al \cite{corbera2011central} provide a condition on the infinitesimal masses in the $(1+3)$-body problem to get symmetric configurations. It is the same condition for the infinitesimal circulations given in \cite{barry2016existence} for the $(1+3)$-vortex problem. Bifurcations of symmetric relative equilibria in the $(1+3)$-body problem are described by Corbera et al \cite{corbera2015bifurcation} and Chen et al \cite{chen2023symmetric}, while rigorous counting of the number of relative equilibria is done by Tsai in \cite{tsai2016counting}. Tsai also rigorously counts the number of relative equilibria in the $(1+3)$-vortex problem in \cite{tsai2017numbers}, improving on results in \cite{hoyer2014bifurcations}.

Albouy and Fu \cite{albouy2009relative} classify all relative equilbria in the $(1+4)$-body problem with four identical infinitesimal masses as squares, kites
, and an isosceles trapezoid. Equivalent numerical results for the $(1+4)$-vortex problem are given in \cite{barry2012relative}. Oliveria \cite{oliveira2013symmetry}, Deng et al \cite{deng2019symmetric}, and Deng et al \cite{deng2022symmetry} extend these results for unequal infinitesimal masses and classify symmetry of configurations by whether the line of symmetry contains any of the infinitesimal masses. They described and find configurations that are similar to the kites, squares, and isosceles trapezoids described in this paper.

Also of note, in many of these papers (\cite{albouy2009relative,long2002four,perez2007convex,lee2009central,oliveira2019stability,marchesin2019global,cabral2000stability,barry2012relative,ohsawa2024nonlinear}), the circulations or masses of the point objects are assumed to be in a some fixed ratio such that the set of parameters can be given by a one-dimensional set. In this paper and in \cite{barry2016existence}, we give noteworthy examples of vortex configurations with a two-parameter family of circulations. 

The remainder of the introduction summarizes the necessary definitions and theorems about relative equilibria in the $(1+N)$-vortex problem, defines the different types of symmetric configurations, and gives a brief background of the techniques from algebraic geometry used to prove the results in this paper. The following sections consider the four types of symmetric configurations: the $(1+4)$-gon, rectangles, trapezoids with three equal sides, and kites. Assuming the type of symmetry given, the necessary ratios of the circulation parameters for the infinitesimal vortices are proved, as well as any restrictions on the positions of the vortices around the unit circle.

\subsection{Relative Equilibria of the ($1+N$)-Vortex Problem}
	The classical $n$-vortex problem is a point vortex differential equations model for $n$ well-separated vortices in a two-dimensional fluid. Let $q_i=(x_i,y_i)\in \mathbb{R}^2$ be the position of the $i$th vortex and let $\Gamma_i$ be the circulation of the $i$th vortex.. The equations of motion for the $n$ vortices are a Hamiltonian system with Hamiltonian $H(q)=-\sum\limits_{i<j} \Gamma_i\Gamma_j \log|q_i-q_j|$ so that 
	\begin{equation}
	\Gamma_i\dot{q}_i=J\nabla_i H(q) \quad \text{ with } \quad J=\begin{bmatrix} 0 & 1 \\ -1 & 0 \end{bmatrix} \label{vortexDE}
	\end{equation}
	and where $\nabla_i$ is the two-dimensional partial gradient with respect to $q_i$. 
	A relative equilibrium is a solution $q_i(t)=e^{-\omega J t}q_i(0), i=1,...,n$ to \eqref{vortexDE} which rotates around the center of vorticity at the origin with angular velocity $\omega$. We will assume $\omega=1$.

We consider the case of one dominant and $N$ smaller vortices. Let $\Gamma_0=1$ be the circulation of one strong vortex and let $\Gamma_i=\epsilon \mu_i$, $\mu_i\neq 0$ for $i=1,...,N$ be the circulations of $N$ smaller vortices. This is sometimes referred to as the $(N+1)$-vortex problem, and we use this convention here. Given a sequence of relative equilibria to the $(N+1)$-vortex problem, parameterized as $\epsilon \to 0$, the limiting case is called a relative equilibria of the $(1+N)$-vortex problem. We consider only relative equilibria of the $(1+N)$-vortex problem where the vortices are bounded away from each other and do not collide in the limit. 

Note that this is different than the restricted problem where the circulations of the infinitesimal vortices are set to zero and so that the infinitesimal vortices are passive under the influence of the strong vortex. By considering the limiting case, we preserve the potency of the interactions between the weaker vortices while taking the limit. For a detailed discussion of relative equilibria of the $(1+N)$-vortex problem and for proofs of the following lemma and theorems, see \cite{barry2012relative} and \cite{barry2016existence}.

\begin{lemma}[Lemma 2 in \cite{barry2016existence}]\label{config lemma}
In the limit as $\epsilon \to 0$, $|q_0|\to 0$ and $|q_i|=1$ for $i=1,...,N$. In other words, in the limit, the strong vortex is at the origin, and the infinitesimal vortices are on the unit circle. 
\end{lemma}

\begin{theorem}[Theorem 1 in \cite{barry2016existence}] \label{thm:V}
Let $(\bar{r},\bar{\theta})=(1,...,1,\bar{\theta}_1,...\bar{\theta}_N)$ be the positions (in polar coordinates) of the $N$ small vortices in a a relative equilibrium of the $(1+N)$-vortex problem. Then $\bar{\theta}$ is a critical point of the function 
 \begin{equation}V(\theta)=-\sum_{i<j} \mu_i \mu_j [ \cos(\theta_i-\theta_j)+\tfrac12 \log(2-2\cos(\theta_i-\theta_j))] \label{Vequation}
 \end{equation}
\end{theorem}

The function $V$ has a few important symmetries. Any rotation of a critical point of $V$ is also a critical point of $V$. $V$ is an even function so if $\bar{\theta}$ is a critical point of $V$, then so is $-\bar{\theta}$, corresponding to reflection over the $x$-axis in the position of vortices around the unit circle. Additionally, if $\bar{\theta}$ is a critical point of $V$ for the parameter set $(\mu_1,\mu_2,\mu_3,\mu_4)$, then $\bar{\theta}$ is also a critical point for any scalar multiple of $(\mu_1,\mu_2,\mu_3,\mu_4)$. 

Because of rotational symmetry, all critical points of $V$ are degenerate, i.e. the Hessian $V_{\theta \theta}$ at the critical point has one zero eigenvalue, corresponding to the eigenvector $v=\begin{bmatrix}1 & 1 & 1 & \dots 1\end{bmatrix}^T$. However, we can partition the nullspace of $V_{\theta \theta}$ into the span of $v$ and its complement, and define nondegeneracy on the complement.
\begin{definition} \label{def:nondegen}
A critical point $\bar{\theta}$ of $V$ is \textit{nondegenerate} provided the Hessian $V_{\theta \theta}(\bar{\theta})$ has only one zero eigenvalue.
\end{definition} 

\begin{theorem} [Theorem 2 in \cite{barry2016existence}] \label{thm:existence}
Suppose $\bar{\theta}=(\bar{\theta}_1,...,\bar{\theta}_N)$ is a nondegenerate critical point of $V$.  Then for $\bar{r}=(1,1,...,1)$, the configuration $(\bar{r},\bar{\theta})$ are the positions of the $N$ infinitesimal vortices in a relative equilibrium of the $(1+N)$-vortex problem.
\end{theorem}

Since $V_{\theta \theta}(\theta,k\mu_1,k\mu_2,k\mu_3,k\mu_4) = k^2V_{\theta \theta}(\theta,\mu_1,\mu_2,\mu_3,\mu_4)$, scaling of the circulation parameters preserves the degeneracy or nondegeneracy of a critical point.

 The function $V$ also works as a sort of potential function for relative equilibria of the $(1+N)$-vortex problem, in that the eigenvalues of a \textit{weighted} Hessian correspond to the linear stability of relative equilibria in the full, not limiting, $(N+1)$-vortex problem. Let $\mu$ be the diagonal matrix with the circulation parameters $\mu_1,\mu_2,...,\mu_N$ on the diagonal. The stability criteria are given in the next theorem.

\begin{theorem}[Theorem 3 in \cite{barry2016existence}] \label{stability theorem} 
Let $(r^{\varepsilon}, \theta^{\varepsilon})$ be a sequence of relative equilibria of the $(N+1)$-vortex problem that converges to a relative equilibrium $(\bar{r}, \bar{\theta})=(1,...,1,\bar{\theta}_1,...,\bar{\theta}_N)$ of the $(1+N)$-vortex problem as $\varepsilon \to 0$, and let $\bar{\theta}$ be a nondegenerate critical point of $V$.
For $\varepsilon$ sufficiently small, $(r^{\varepsilon},\theta^{\varepsilon})$ is nondegenerate and is linearly stable if and only if $\mu^{-1}V_{\theta\theta}(\bar{\theta})$ has $N-1$ positive eigenvalues.
\end{theorem}

Since $\mu^{-1}V_{\theta\theta}(\theta,k\mu_1,k\mu_2,k\mu_3,k\mu_4) = k\mu^{-1}V_{\theta\theta}(\theta,\mu_1,\mu_2,\mu_3,\mu_4)$, a relative equilibrium that is linearly stable for the parameter set $(\mu_1,\mu_2,\mu_3,\mu_4)$ will also be linearly stable for parameter sets $k(\mu_1,\mu_2,\mu_3,\mu_4)$ where $k>0$. Additionally, a nondegenerate critical point of $V$ that has $N-1$ \textit{negative} eigenvalues for the parameter set $(\mu_1,\mu_2,\mu_3,\mu_4)$ will continue to a to linearly stable relative equilibrium for $k(\mu_1,\mu_2,\mu_3,\mu_4)$ when $k<0$. 

\subsection{Symmetry when $N=4$} \label{sec:symmetry defs}

\begin{figure}[h]
\centering
\begin{subfigure}[t]{.23\textwidth}
\includegraphics[width=1.4in]{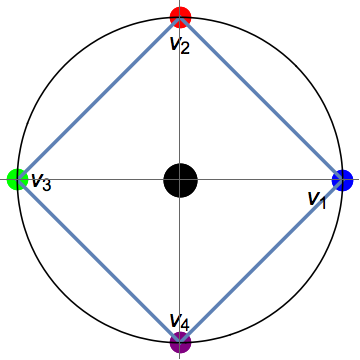}
\caption{The $(1+4)$-gon}
\label{squaredef}
\end{subfigure}
\begin{subfigure}[t]{.23\textwidth}
\includegraphics[width=1.4in]{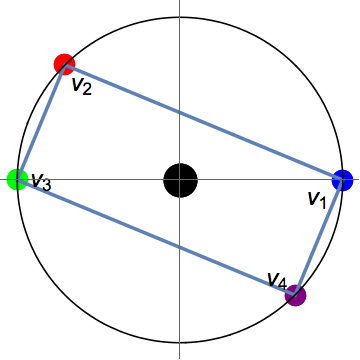}
\caption{A rectangle}
\label{rectangledef}
\end{subfigure} 
\vspace{.2in}
\begin{subfigure}[t]{.23\textwidth}
\includegraphics[width=1.4in]{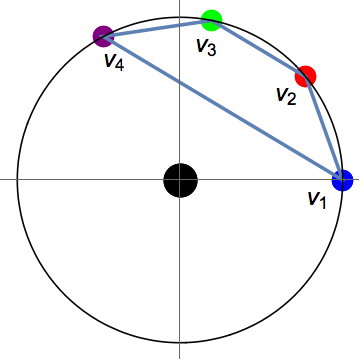}
\caption{A trapezoid with three equal sides}
\label{threeequaldef}
\end{subfigure} 
\begin{subfigure}[t]{.23\textwidth}
\includegraphics[width=1.4in]{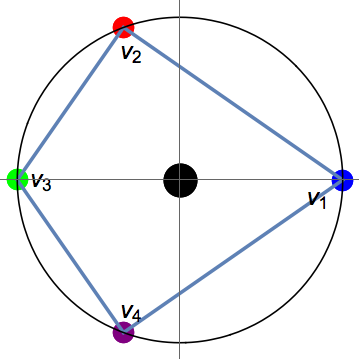}
\caption{A Kite }
\label{kitedef}
\end{subfigure}  
\caption{Examples of symmetric relative equilibria in the $(1+4)$-vortex problem}
\label{definition pictures}
\end{figure}

We consider the case of relative equilibria of the $(1+N)$-vortex problem when $N=4$. As stated in Lemma \ref{config lemma}, the dominant vortex is at the origin, and the four infinitesimal vortices are positioned around the unit circle. The five vortices together may form a convex or concave configuration, and the overall configuration will be convex only when all four infinitesimal vortices lie on one half of the circle with the dominant vortex at the origin. Instead, we consider convex quadrilaterals inscribed in the unit circle by the four infinitesimal vortices, with the vertices of the quadrilateral as the positions of the vortices.

Symmetry is defined by reflective symmetry of the quadrilateral. There are two cases depending on whether the line of symmetry contains any vortices. A kite has two vortices on an axis of symmetry while the other two vortices are symmetric by reflection over this axis.  All other symmetric configurations inscribe isosceles trapezoids and have a line of reflective symmetry containing no vortices. We only consider subcases of isosceles trapezoids with one degree of freedom: rectangles and trapezoids with three equal sides. 

At the intersection of the definitions of rectangles, trapezoids with three equal sides, and kites is a square. A regular polygon inscribed in a circle with central vortex is called the $(1+N)$-gon configuration. When $N=4$, this is, of course, a square.

Let $v_1,v_2,v_3,v_4$ be the four infinitesimal vortices. The position of the vortex is given by the angle around the unit circle, so when the value of $\theta_i$ is defined, this means that $v_i$ is at the point $(1,\theta_i)$ in polar coordinates.

In the proofs that follow, we set $\theta_1=0$ to reduce the rotational symmetry of the critical points of $V$. For the $(1+4)$-gon, assuming the ordering of the vortices around the unit circle counterclockwise is $v_1<v_2<v_3<v_4$, we can assume that $\theta_2=\pi/2$, $\theta_3=\pi$, $\theta_4=3\pi/2$.

The other configurations are defined with one degree of freedom.  For rectangles, pairs of vortices lie on parallel lines, so with $\theta_1=0$, we fix $\theta_3=\pi$, with $\theta_2$ free and $\theta_4=\theta_2+\pi$. For trapezoids with three equal sides, we let $\theta_2$ be free with $\theta_3=2\theta_2$, and $\theta_4=3\theta_2$. For kites, we assume the line of reflection is over the $x$-axis and fix $\theta_3=\pi$. Then $\theta_2$ is free and $\theta_4=-\theta_2$.

Figure \ref{definition pictures} shows the four types of symmetric configurations considered in this paper.

\subsection{Gr\"{o}bner Bases and Elimination Ideals}

We give a brief overview of the techniques from algebraic geometry used in this paper. A \textit{What is} paper by Sturmfels \cite{sturmfels2005grobner} gives a brief and insightful introduction to Gr\"{o}bner bases. For a deeper look, see the book \textit{Using Algebraic Geometry} by Cox, Little, and O'Shea \cite{cox2006using}.

At its simplest, a Gr\"{o}bner basis is a technique for solving a system of polynomial equations. The algorithm involved finds another set of polynomials with the same set of zeros as the original equations. The advantages of a Gr\"{o}bner basis come in the monomial ordering in the algorithm, naturally ordering the new polynomials from simplest (fewest variables and lowest exponents) to more complicated. 

Let $k$ be a field, and let $\mathcal{P}=\{p_1,...,p_i\}$ be a set of polynomials in the polynomial ring $k[x_1,...,x_n]$. Then $\mathcal{P}$ generates an ideal
\begin{equation*}
\langle \mathcal{P} \rangle = \{h_1p_1+...+h_ip_i \text{ where } p_1,...,p_i \in \mathcal{P} \text{ and } h_1,...,h_i \in k[x_1,...,x_n] \}
\end{equation*}

The set of zeros of $\mathcal{P}$ is the \textit{variety} of $\mathcal{P}$ and denoted $Var(\mathcal{P})$. The variety of a set of polynomials and the variety of the ideal it generates are the same, $Var(\mathcal{P})=Var(\langle \mathcal{P} \rangle)$. The set $\mathcal{P}$ is referred to as a basis for the ideal it generates. A Gr\"{o}bner basis $\mathcal{G}$ of $\langle \mathcal{P} \rangle$ is another basis with specific properties for the same ideal, so that $Var(\mathcal{G})=Var(\mathcal{P})$. 

A Gr\"{o}bner basis is calculated algorithmically. In this paper, Gr\"{o}bner bases are implemented in Mathematica 12 using the default Gr\"{o}bner basis algorithm (a Gr\"{o}bner walk rather than the historical Buchberger's Algorithm). However the calculation is dependent on the choice of ordering of the variables. The ordering must be a total well-ordering that preserves multiplication on the variables. In this paper, we use two monomial orderings in Mathematica, the default \texttt{Lexicographic} ordering or the \texttt{DegreeReverseLexicographic} ordering, which is equivalent to the graded reverse lexicographic ordering (\textit{grevlex}) as defined in \cite{cox2006using}. The ordering used in each Gr\"{o}bner basis calculation is specified throughout the paper. 

Additionally, the Gr\"{o}bner basis for an elimination ideal can be used to eliminate variables completely. For an ideal $I \subset k[x_1,...,x_k,x_{k+1},...,x_n]$, an \textit{elimination ideal} is $I_k=I \cap k[x_{k+1},...,x_n]$. This eliminates the first $k$ variables from the ideal. Additionally, if $\mathcal{G}$ is the Gr\"{o}bner basis for $I$, then the Gr\"{o}bner basis for the elimination ideal $I_k$ is $G\cap I_k$. Geometrically, this is equivalent to projecting the variety $Var(I)$ onto $x_{k+1}\cdots x_n$-space. 

\begin{example}\label{GBelimination example}
To illustrate the above ideas, we give an overly simple example. Consider the set of polynomials
\begin{equation*}
\mathcal{P} = \{ x-y-z+2,x^2+y^2-z \} 
\end{equation*}
The sets of roots of these polynomials are a plane and a paraboloid in $\mathbb{R}^3$, respectively. The variety of $\mathcal{P}$ is the intersection of these two surfaces. In Mathematica, the command
\begin{center}
\texttt{GroebnerBasis[P,\{x,y,z\}]}
\end{center}
 gives a Gr\"{o}bner basis $\{4 - 4 y + 2 y^2 - 5 z + 2 y z + z^2, 2 + x - y - z\}$ in the default lexicographic ordering, while
 \begin{center} \texttt{GroebnerBasis[P,\{x,y,z\},MonomialOrder$\to$DegreeReverseLexicographic]}\end{center}
  gives a Gr\"{o}bner basis $\{2 + x - y - z, 4 - 4 y + 2 y^2 - 5 z + 2 y z + z^2\}$ in \textit{grevlex} ordering .
Finally, we can find the projection of the variety onto the $xy$-plane by computing the Gr\"{o}bner basis of the elimination ideal using a \textit{grevlex} or Degree Reverse Lexicographic ordering. In Mathematica, the command is
\begin{center}
\texttt{GroebnerBasis[P,\{x,y\},z,MonomialOrder$\to$EliminationOrder]}\end{center}
which results in the Gr\"{o}bner basis $\{-2 - x + x^2 + y + y^2\}$.

The surfaces described by the polynomials of $\mathcal{P}$, the variety $Var(\mathcal{P})$, and the projection of $Var(\mathcal{P})$ are shown in Figure \ref{GBex}.

\end{example}

Often we will want to eliminate trivial or unwanted subsets of a variety. See Example \ref{GBsaturation example} where we remove a plane from the variety in order to explicitly identify the rest of the structure.
One common trick is to \textit{saturate the ideal} before a Gr\"{o}bner basis calculation. Assume the roots of a polynomial $g$ are a subset of $Var(\mathcal{P})$, but we would like find a basis for the ideal with variety $Var(\mathcal{P})\setminus \{g=0\}$. We add the polynomial $wg+1$ to the set $\mathcal{P}$ and calculate the Gr\"{o}bner Basis while eliminating the variable $w$. This eliminates the roots of $g$ from the variety of the saturated ideal. 

\begin{example}\label{GBsaturation example}
Consider the set of polynomials
\begin{equation*}
\mathcal{Q} = \{ x^2-xy-xz+2x,x^3+xy^2-xz \} 
\end{equation*}
This is almost the same set of polynomials as in Example \ref{GBelimination example}, but now both polynomials have been multiplied by $x$, so that the variety of $\mathcal{Q}$ consists of the ellipse $Var(\mathcal{P})$ and the plane $x=0$. We can remove the trivial root $x=0$ by saturating the ideal with the polynomial $wx+1$. Define the set $\mathcal{Q'}=\{ x^2-xy-xz+2x,x^3+xy^2-xz, wx+1\}$ and compute the Gr\"{o}bner basis for the elimination ideal with monomial ordering $w>x>y>z$ and eliminate the variable $w$. In Mathematica, the command is
 \begin{center} \texttt{GroebnerBasis[Qprime,\{x,y,z\},w,MonomialOrder$\to$EliminationOrder]}\end{center}
which results in the Gr\"{o}bner basis $\{2 + x - y - z, 4 - 4 y + 2 y^2 - 5 z + 2 y z + z^2\}$, which is the Gr\"{o}bner basis for $\mathcal{P}$ in $grevlex$ ordering. 
\end{example}

\subsection{Sturm's Algorithm for counting roots of a polynomial}

Theorems \ref{thm:3equalsides} and \ref{thm:kite configs} depends on Sturm's Algorithm for counting roots of a polynomial in one variable. Let $p(x)$ be a degree $n$ polynomial. Construct the Sturm sequence 
\begin{align}
p_0(x)&=p(x)\\
p_1(x)&=p'(x)\\
p_{i}(x)&=-\text{Rem}(p_{i-2},p_{i-1})
\end{align}
The sequence is at most length $n$. We evaluate this sequence at some value of $x$ and record the number of sign changes in the sequence as $v(x)$. When $x\to +\infty$ the sign of the polynomial is determined by the sign of the leading coefficient. When $x\to-\infty$, the sign of the polynomial is the sign of the leading coefficient if the polynomial has even degree and is opposite of the sign of the leading coefficient if the polynomial had odd degree.
\begin{theorem}[Sturm]
The total number of distinct real roots of $p(x)$ in the interval $(a,b]$ is $v(a)-v(b)$.
\end{theorem} 

In Theorem \ref{thm:kite configs}, the polynomial in question depends on three parameters. Sturm's theorem is applicable for polynomials with parameters, though it is applied through a series of branches determined by whether the leading coefficient of each new polynomial in the sequence is zero or nonzero. For more details, see sections 1.3.2 and 1.3.3 in \cite{Coste}.


%

\begin{figure}[h]
\centering
\includegraphics[width=.35\textwidth]{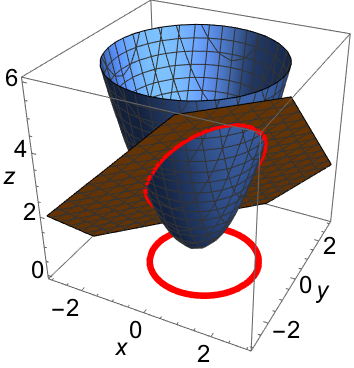}
\caption{Pictured are the surfaces in $\mathcal{P}$ for Example \ref{GBelimination example}, the variety $Var(\mathcal{P})$ and its projection onto the $xy$-plane, i.e. the variety of the elimination ideal which eliminates the variable $z$.}
\label{GBex}
\end{figure}

	\subsection{Outline of Proofs}\label{process}	
In each of the following sections, a certain type of symmetry is assumed and we derive conditions on the circulation parameters necessary to get that symmetry. Each proof uses Mathematica for algebraic simplification, converting the gradient $\nabla V(\theta)$ into an equivalent system of polynomials and then computing a Gr\"{o}bner basis.\footnote{The authors can share the Mathematica notebooks for the work in this paper upon request.}  The process in Mathematica is the following
\begin{enumerate}
\item We start by calculating $\nabla V$ and dividing each function by the common factor $\mu_i$.
	\begin{align}
	\mu^{-1}\nabla V &= (V_{\theta_1},V_{\theta_2},V_{\theta_3},V_{\theta_4})\\
	V_{\theta_i}&= \sum_{j\neq i} \mu_j \left(-\sin(\theta_i - \theta_j) + \frac{\sin(\theta_i - \theta_j)}{2 - 2 \cos(\theta_i- \theta_j)}\right)
	\end{align}
Note 1: We will be using a slight abuse of notation where $V_{\theta_i}=\dfrac{1}{\mu_i} \dfrac{\partial V}{\partial \theta_i}$.

Note 2: The equations of $\mu^{-1}\nabla V$ are linearly dependent: $\mu_1V_{\theta_1}+\mu_2V_{\theta_2}+\mu_3V_{\theta_3}+\mu_4V_{\theta_4}=0.$ 
Thus we need only find solutions to the system 
\begin{equation} V_{\theta_2}=0, \quad V_{\theta_3}=0, \quad V_{\theta_4}=0 \end{equation}

\item Next, the symmetry assumptions are substituted into the equations $V_{\theta_i}$. We set $\theta_1=0$ to reduce by rotational symmetry in $V$. In each case, the only angular variable left in the equations is $\theta_2$.

\item The fractions of each function of $V_{\theta_i}$ are given a common denominator using the \texttt{Together} command. Since we are looking for solutions to $\mu^{-1}\nabla V=0$, we consider only the numerators of these equations using the \texttt{Numerator} command. The resulting equations are referred to as $V_{\theta_i}^{num}$. Note that the denominators of $V_{\theta_i}$ are \begin{equation*}2 (-1 + \cos(\theta_1 - \theta_4)) (-1 + \cos(\theta_2 - \theta_4)) (-1 + \cos(\theta_3 - \theta_4)),\end{equation*} the roots of which correspond to collisions of the vortices and can be ignored. 

\item The trigonometric terms are expanded using trigonometric identities by the \texttt{TrigExpand} command in Mathematica. These trigonometric identities can introduce some fractions with constant denominators, so we again apply the \texttt{Together} and \texttt{Numerator} commands. The equations are factored again using the \texttt{Factor} command, and any roots corresponding to collisions are removed. These collision roots are $\sin(\theta_2)$, corresponding to a collision of the second vortex with the first vortex with $\theta_1=0$ or with the third vortex in the cases when we assume $\theta_3=\pi$.

\item The trigonometric terms are replaced using the tangent half angle identities
\begin{equation} \sin(\theta_2)=\frac{2r}{1+r^2} \qquad \cos(\theta_2)=\frac{r^2-1}{1+r^2} \label{eq:tangent_half_angle}\end{equation}
These identities are such that $r=0$ corresponds to $\theta_2=\pi$, $r>0$ corresponds to $\theta_2 \in (0,\pi)$, and $r<0$ corresponds to $\theta_2 \in (-\pi,0)$. Furthermore, these corresponcies are symmetric since $r^*$ corresponds to $\theta_2^*$ if and only if $-r^*$ corresponds to $-\theta_2^*$. 

\item Again the equations are factored using the \texttt{Factor} command and the denominators are discarded using the \texttt{Numerator} command. Factors that correspond to collisions of vortices are removed from the equations. These factors are powers of $r$ (corresponding to $\theta_2=\theta_3=\pi$ in the kites and rectangles case), and $-1+3r^2$ (corresponding to $\theta_2=2\pi/3$, implying $\theta_4 \mod 2\pi =\theta_1=0$ in the trapezoids with three equal sides case). At this point we have polynomial equations whose roots correspond to the zeros of $\mu^{-1}\nabla V=0$. These polynomial equations are referred to as $V_{\theta_i}^{num*}$.

\item A Gr\"{o}bner basis is calculated for the ideal generated by $V_{\theta_i}^{num*}$ or for a saturated ideal generated by $V_{\theta_i}^{num*}$ and other polynomials. The monomial and elimination orderings used are given in each proof. 

\item Any additional analysis on the Gr\"{o}bner basis concludes the proof.
\end{enumerate}

\section{The $1+N$-gon with $N=4$}\label{section:square}
The infinitesimal vortices in the $(1+4)$-gon inscribe a square in the unit circle.

\begin{theorem}\label{thm:square} Assuming $\mu_i\neq0$, the $(1+4)$-gon must have two equal pairs of vortices at opposite corners of the square. \end{theorem}

\begin{proof}
Let $\theta_1=0$, $\theta_2=\pi/2$, $\theta_3=\pi$, and $\theta_4=3\pi/2$  so that the vortices are ordered $v_1<v_2<v_3<v_4$ as the corners of a square inscribed in the unit cirlce. Then $v_1$ is opposite $v_3$ and $v_2$ is opposite $v_4$. Substituting these values of $\theta_i$ into $V_{\theta_i}^{num}$ gives 
\begin{equation}\tfrac12\{(\mu_4-\mu_2),(\mu_1-\mu_3), (\mu_2-\mu_4),(\mu_3-\mu_1)\}.\end{equation} The zero set of these polynomials gives $\mu_1=\mu_3$ and $\mu_2=\mu_4$. Thus the vortices on opposite corners must have equal circulations. Because of the symmetry of $V(\theta)$, any permutation of $v_1,v_2,v_3,v_4$ will result in vortices on opposite corners having equal circulations. \end{proof}

This theorem gives sufficient but not necessary conditions for a $(1+4)$-gon. Two equal pairs of vortices can also give kite (see Theorem \ref{thm:kite configs}) and anti-symmetric configurations of the vortices. 
However, with two equal pairs and the assumption of rectangular or trapezoid with three equal sides type symmetry, the configuration is necessarily a square. See Lemmas \ref{lem:recs_squares} and \ref{lem:traps_squares}.

\begin{theorem} \label{square stability}
The $(1+4)$-gon is a nondegenerate critical point of $V(\theta)$ except when the pairs of vortices are in a $3:2$ ratio, and are never linearly stable, using the criteria in Theorem \ref{stability theorem}.
\end{theorem} 

\begin{proof}
We substitute $\theta_1=0$, $\theta_2=\pi/2$, $\theta_3=\pi$, and $\theta_4=3\pi/2$ and $\mu_3=\mu_1$, $\mu_4=\mu_2$ into the Hessian $V_{\theta\theta}$, and compute the eigenvalues using Mathematica: 
\begin{equation*}\lambda = 0, 2 \mu_1 \mu_2, \tfrac12 (-3 \mu_1^2 + 2 \mu_1 \mu_2), \tfrac12 (2 \mu_1 \mu_2 - 3 \mu_2^2) \end{equation*}
When $\mu_2=\frac32\mu_1$ or $\mu_2=\frac23\mu_1$, there is more than one zero eigenvalue and the critical point is degenerate.

To determine linear stability, we calculate the eigenvalues of the weighted Hessian $\mu^{-1}V_{\theta\theta}$ in Mathematica:
\begin{equation*} \lambda= 0, \tfrac12 (2 \mu_1 - 3 \mu_2), \tfrac12 (-3 \mu_1 + 2 \mu_2), \mu_1 + \mu_2 \end{equation*}
Linear stability criterion in Theorem \ref{stability theorem} requires $N-1$ positive eigenvalues of $\mu^{-1}V_{\theta\theta}$. Using \texttt{Reduce} in Mathematica on the equations \begin{equation*}  \tfrac12 (2 \mu_1 - 3 \mu_2)>0, \quad \tfrac12 (-3 \mu_1 + 2 \mu_2)>0, \quad  \mu_1 + \mu_2>0 \end{equation*} 
produces a \texttt{False} result. In other words, it is never possible for all three nonzero eigenvalues to be positive, and the $(1+4)$-gon is never linearly stable.  \hfill $\square$\end{proof} 

Note that when $\mu_i=1$ for all $i$ or when $\mu_i=-1$ for all $i$, our results coincide with the stability results of the $(1+N)$-gon in \cite{barry2012relative}.

\FloatBarrier
\section{Rectangles}
	We use the property that rectangles are symmetric by reflection through their centroid. By inscribing the rectangle in the unit circle, we put the centroid at the origin. Assume that vortices $v_1$ and $v_3$ are symmetric by reflection through the origin, as is the pair $v_2$ and $v_4$.
	
	\begin{lemma} \label{lem:recs_squares}
	Assume that the positions of the vortices $v_1$ and $v_3$ are symmetric by reflection through the origin, as are the positions of the pair $v_2$ and $v_4$. If $\mu_1=\mu_3$ and $\mu_2=\mu_4$, $\mu_i \neq 0$ for $i=1,2,3,4$, then the vortices form the square configuration. 
\end{lemma}

\begin{proof} Without loss of generality, assume $\theta_1=0$, and the vortices are symmetric by reflection through the origin so that $v_3$ is the reflection $v_1$. Then $\theta_3=\pi$. Assume $\theta_2\in(-\pi,\pi)\setminus\{ 0 \}$ and $\theta_4=\theta_2+\pi$. Configurations of this type can be described by a one parameter family $\theta=(0,\theta_2,\pi,\theta_2+\pi)$. Additionally assume $\mu_1=\mu_3$ and $\mu_2=\mu_4$. Substituting these values into $\mu^{-1}\nabla V$ and simplifying as in Section \ref{process} gives the following
\begin{align}
V_{\theta_2}^{num} & = 2 \mu_1 \cos(\theta_2)\sin(\theta_2) \\
V_{\theta_3}^{num} & = -2\mu_2 \cos(\theta_2) \sin(\theta_2)\\
V_{\theta_4}^{num} & = 2\mu_1 \cos(\theta_2) \sin(\theta_2)
\end{align}
We ignore the factor of $\sin(\theta_2)$ which corresponds to collisions of $v_2$ with $v_1$ or $v_3$, and consider roots for which $\mu_1, \mu_2 \neq 0$, we see that $\cos(\theta_2)=0$ so that $\theta_2 = \pm \pi/2$, and the configuration must be a square. \hfill $\square$
\end{proof}
	
	\begin{theorem}\label{thm:rectangles}
		Assume that the positions of the vortices $v_1$ and $v_3$ are symmetric by reflection through the origin, as is the pair $v_2$ and $v_4$. Then the only rectangular configuration that is not the square occurs when the lines containing each pair are at an angle of $\frac{\pi}{4}$ and when the circulations satisfy $\mu_1=-\mu_3$ and $\mu_2=-\mu_4$, $\mu_i\neq 0$ for $i=1$ to $4$.
	\end{theorem}

\begin{proof} Without loss of generality, assume $\theta_1=0$, and the vortices are symmetric by reflection through the origin so that $v_3$ is the reflection $v_1$. Then $\theta_3=\pi$. Assume $\theta_2\in(-\pi,\pi)\setminus\{ 0\}$ and $\theta_4=\theta_2+\pi$. Configurations of this type can be described by a one parameter family $\theta=(0,\theta_2,\pi,\theta_2+\pi)$. Additionally assume $\mu_1\neq \mu_3$ and $\mu_2\neq \mu_4$.

Substituting $\theta_1=0,\, \theta_3=\pi$ and $\theta_4=\theta_2+\pi$ into $\mu^{-1}\nabla V$ and simplifying as described in Section \ref{process} gives the following
\begin{align}
V_{\theta_2}^{num}&= ((\mu_1+\mu_3) \cos(\theta_2) + (\mu_1-\mu_3) \cos(2\theta_2)) \sin(\theta_2)\\
V_{\theta_3}^{num}&= - ((\mu_2+\mu_4) \cos(\theta_2) +(-\mu_2+\mu_4) \cos(2\theta_2)) \sin(\theta_2) \\
V_{\theta_4}^{num}&=((\mu_1 + \mu_3) \cos(\theta_2) + (\mu_3 - \mu_1) \cos(2\theta_2)) \sin(\theta_2) 
\end{align}
After removing the $\sin(\theta_2)$ factor from all three equations, expanding trig functions and using the tangent half-angle change of variables, we get the polynomials 
\begin{align}
V_{\theta_2}^{num^*}&=4(-\mu_3 - 3 \mu_1 r^2 + 3 \mu_3 r^2 + \mu_1 r^4)\\
V_{\theta_3}^{num^*}&=4( \mu_2 - 3 \mu_2 r^2 + 3 \mu_4 r^2 - 
 \mu_4 r^4)\\
V_{\theta_4}^{num^*}&=4(-\mu_1 + 3 \mu_1 r^2 - 3 \mu_3 r^2 + \mu_3 r^4)
\end{align}

We will now saturate the ideal by including the polynomials $\nu(\mu_1-\mu_3)+1$ and $\omega(\mu_2-\mu_4)+1$ so that $\mu_1=\mu_3$ and $\mu_2=\mu_4$ are not included in the ideal, and we eliminate the possibility of a square configuration. We project the ideal $\langle V_{\theta_2}^{num}\textsuperscript{*},V_{\theta_3}^{num}\textsuperscript{*}, V_{\theta_4}^{num}\textsuperscript{*}, \nu(\mu_1-\mu_3)+1, \omega(\mu_2-\mu_4)+1\rangle$ on to the $(\mu_1,\mu_2,\mu_3,\mu_4,r)$-space. We specify the elimination ordering $\nu>\omega>\mu_1>\mu_2>\mu_3>\mu_4>r$ and use a Gr\"{o}bner basis to calculate the elimination ideal in \textit{grevlex} ordering that removes the extra variables $\nu$ and $\omega$. We perform Gr\"{o}bner basis algorithm through Mathematica to obtain the following basis for the elimination ideal
 \begin{equation} \{ \mu_2+\mu_4, \mu_1+\mu_3,1-6r^2+r^4\}
\end{equation}
Thus rectangular configurations that are not squares must satisfy $\mu_2=-\mu_4$, $\mu_1=-\mu_3$, and $1-6r^2+r^4=0$. This last equation gives the possible positions of $v_2$ and $v_4$. The roots of $1-6r^2+r^4=0$ are $r=1\pm \sqrt{2}, -1\pm \sqrt{2}$, corresponding to $\theta_2=\pm \pi/4, \pm 3\pi/4$. \hfill $\square$
\end{proof}

\begin{theorem} \label{recstability} Rectangular configurations that are not squares are nondegenerate critical points of $V(\theta)$ and always linearly unstable. \end{theorem}

The proof of Theorem \ref{recstability} follows from substituting in $\mu_1=-\mu_3$, $\mu_2=-\mu_4$ and $\theta_1=0$, $\theta_3=\pi$, $\theta_4=\theta_2+\pi$ and $\theta_2 =\pm \pi/4, \pm 3\pi/4$ into the matrices $V_{\theta \theta}$ and $\mu^{-1}V_{\theta \theta}$ and using the \texttt{Eigenvalues} command in Mathematica. There are three nonzero eigenvalues of $V_{\theta \theta}$ for $\mu_1, \mu_2\neq 0$, and $\mu^{-1}V_{\theta \theta}$ has two zero eigenvalues, possibly indicating a bifurcation in the critical points of $V$.

\section{Trapezoids with three equal sides} \label{section:trap3equal}
	
An inscribed trapezoid with three equal sides with ordering $v_1,v_2,v_3,v_4$ of vortices around the unit circle can be parameterized as $\theta_1=0$, $\theta_2 \in (0,2\pi/3)$, $\theta_3=2\theta_2$, and $\theta_4=3\theta_2$. This trapezoid would have a line of reflection at an angle of $\frac{3}{2}\theta_2$. We start with a more general assumption of $\theta_2 \in (-\pi,\pi)\setminus \{\pm 2\pi/3\}$ (both of these values giving a collision between $v_1$ and $v_4$). 

\begin{lemma} \label{lem:traps_squares}
Assume $\mu_i\neq 0$ and assume $\theta_1=0$, $\theta_2-\theta_1=\theta_3-\theta_2=\theta_4-\theta_3$, i.e. $\theta_3=2\theta_2$ and $\theta_4=3\theta_2$, with $\theta_2 \in (-\pi,\pi)\setminus \{\pm 2\pi/3\}$. If $\mu_1=\mu_3$ and $\mu_2=\mu_4$, $\mu_i \neq 0$ for $i=1,2,3,4$, then the vortices form the square configuration. 
\end{lemma}

\begin{proof}
Let $\theta_1=0$, $\theta_3=2\theta_2$, and $\theta_4=3\theta_2$, $\theta_2\in(-\pi,\pi)\setminus \{\pm 2\pi/3\}$, and $\mu_1=\mu_2$ and $\mu_3=\mu_4$, as described in the lemma statement. Substituting these values into $\mu^{-1}\nabla V$, removing factors of $\sin(\theta_2)$ which would lead to collisions, we get equations $V_{\theta_2}^{num},V_{\theta_3}^{num},V_{\theta_4}^{num}$ in $\mu_1$, $\mu_2$, and $\theta_2$. Changing coordinates as in Section \ref{process}, we get polynomial equations:
\begin{align}
V_{\theta_2}^{num^*}& = 8 \mu_2 (-1 + r) (1 + r) (1 - 4 r + r^2) (1 + 4 r + r^2)\\
V_{\theta_3}^{num^*}& = -8 \mu_1 (-1 + r) (1 + r) (1 - 4 r + r^2) (1 + 4 r + r^2)\\
\begin{split}
V_{\theta_4}^{num^*}& = 32 (-1 + r) (1 + r) (-1 + 3 r^2) (-\mu_2 - 24 \mu_1 r^2 + 16 \mu_2 r^2 + 
   152 \mu_1 r^4 - 26 \mu_2 r^4\\& \hspace*{.15in}- 72 \mu_1 r^6 - 40 \mu_2 r^6 + 8 \mu_1 r^8 + 
   3 \mu_2 r^8) \end{split}
\end{align}
We see $r=\pm 1$ are two roots of these equations, but we will use a Gr\"{o}bner Basis to consider others. Additionally, we want roots where $\mu_i\neq 0$, so we saturate the ideal by adding two polynomials $\nu\mu_1+1$ and $\omega \mu_2+1$. 
We project the variety of the ideal $\langle V_{\theta_2}^{num}\textsuperscript{*},V_{\theta_3}^{num}\textsuperscript{*}, V_{\theta_4}^{num}\textsuperscript{*}\rangle$ on to $(\mu_1,\mu_2,r)$-space. We specify an elimination ordering $\nu > \omega > \mu_1> \mu_2> r$ and use a Gr\"{o}bner basis to calculate the elimination ideal in \textit{grevlex} ordering that removes the extra variables $\nu$ and $\omega$. We perform Gr\"{o}bner basis algorithm through Mathematica with the elimination ordering to obtain a basis $\{-1+r^2\}$ for the elimination ideal. Thus $r=\pm 1$, corresponding to $\theta_2=\pi/2$ or $-\pi/2$. Thus the only nontrivial configurations of trapezoids with three equal sides and circulation parameters $\mu_1=\mu_3$, $\mu_2=\mu_4$ is the square configuration. \hfill $\square$
\end{proof}

\begin{theorem}\label{thm:3equalsides}
Assume $\mu_i\neq 0$ and assume $\theta_1=0$, $\theta_2-\theta_1=\theta_3-\theta_2=\theta_4-\theta_3$, i.e. $\theta_3=2\theta_2$ and $\theta_4=3\theta_2$, with $\theta_2 \in (-\pi,\pi)\setminus\{\pm 2\pi/3\}$. Additionally, assume that the vortices do not form a square, i.e. $\mu_1\neq \mu_3$ and $\mu_2\neq \mu_4$. Then there are six possible values of $\theta_2$ that satisfy this symmetry, corresponding to the six real roots of the polynomial \begin{equation}
g(r)=-1+33r^2-202r^4+146r^6-117r^8+13r^{10}.\end{equation}
The corresponding values of $\theta_2$ and the set of circulation parameters are given by $\theta_2=\arcsin(2r/(1+r^2))$  and $\{(b(r)\mu_2 + a(r)\mu_3, \mu_2,\mu_3,a(r)\mu_2 +b(r) \mu_3 ) : \mu_2,\mu_3 \in \mathbb{R}\}$, where 
\begin{align} a(r) &  =\tfrac1{272}(-753 + 16124 r^2 - 11428 r^4 + 10036 r^6 - 1131 r^8) \\
 b(r) & = \tfrac{1}{4352}(2827 - 168410 r^2 + 119616 r^4 - 107718 r^6 + 12181 r^8).
 \end{align}
\end{theorem}

\begin{proof}
Let $\theta_1=0$, $\theta_3=2\theta_2$, and $\theta_4=3\theta_2$, $\theta_2\in(-\pi,\pi)\setminus\{\pm 2\pi/3\}$, as described in the theorem statement. Substituting these values into $\mu^{-1}\nabla V$, and changing coordinates as in Section \ref{process}, we get equations $V_{\theta_2}^{num^*},V_{\theta_3}^{num^*},V_{\theta_4}^{num^*}$. The equation associated with $V_{\theta_4}^{num^*}$ contains the factor $(-1 + 3 r^2)$, which we disregard because in this parameterization $r=\pm 1/\sqrt{3}$ corresponds to $\theta_2=\pm 2\pi/3$ and $\theta_4 \text{ mod } 2\pi = \theta_1=0 $. We consider the polynomial equations
\begin{align}
V_{\theta_2}^{num^*}&=-8(\mu_4 - 6 \mu_1 r^2 + 6 \mu_3 r^2 - 15 \mu_4 r^2 - 4 \mu_1 r^4 + 4 \mu_3 r^4 + 
 15 \mu_4 r^4 + 2 \mu_1 r^6 - 2 \mu_3 r^6 - \mu_4 r^6) \label{3equal-pre-GBpolys-1} 
 \\
V_{\theta_3}^{num^*}&= -8(-\mu_1 + 15 \mu_1 r^2 - 
 6 \mu_2 r^2 + 6 \mu_4 r^2 - 15 \mu_1 r^4 - 4 \mu_2 r^4 + 4 \mu_4 r^4 + 
 \mu_1 r^6 + 2 \mu_2 r^6 - 
 2 \mu_4 r^6)  \label{3equal-pre-GBpolys}\\
 \begin{split}
V_{\theta_4}^{num^*}&= 32(\mu_2 + 18 \mu_1 r^2 - 17 \mu_2 r^2 + 6 \mu_3 r^2 - 
   168 \mu_1 r^4 + 42 \mu_2 r^4 - 8 \mu_3 r^4 + 252 \mu_1 r^6 \\
   &\hspace*{.15in}+ 14 \mu_2 r^6 - 
   28 \mu_3 r^6 - 72 \mu_1 r^8 - 43 \mu_2 r^8 - 8 \mu_3 r^8 + 2 \mu_1 r^{10} + 
   3 \mu_2 r^{10} + 6 \mu_3 r^{10}) \label{3equal-pre-GBpolys-3} 
   \end{split}
\end{align}

We will now saturate the ideal by including the polynomials $\nu(\mu_1-\mu_3)+1$ and $\omega(\mu_2-\mu_4)+1$, so that $\mu_1=\mu_3$ and $\mu_2=\mu_4$ are not included in the ideal and we eliminate the possibility of a square configuration. Additionally we saturate the ideal by including the polynomial $z\mu_1+1$, which eliminates the possibility that $\mu_1=0$. (One can also choose $z\mu_4+1$ to get the same result.) We project the ideal $\langle V_{\theta_2}^{num}\textsuperscript{*},V_{\theta_3}^{num}\textsuperscript{*}, V_{\theta_4}^{num}\textsuperscript{*}, \nu(\mu_1-\mu_3)+1, \omega(\mu_2-\mu_4)+1, z\mu_1+1\rangle$ on to the $(\mu_1,\mu_2,\mu_3,\mu_4,r)$-space. We specify the elimination ordering $\nu>\omega>z>\mu_1>\mu_4>\mu_3>\mu_2>r$ and use a Gr\"{o}bner basis to calculate the elimination ideal in \textit{lex} ordering that removes the extra variables $\nu$, $\omega$, and $z$. We perform Gr\"{o}bner basis algorithm through Mathematica to obtain the following basis of three polynomials for the elimination ideal:
\begin{equation} \{ g(r), f_1 (\mu_2,\mu_3,\mu_4,r) , f_2(\mu_1,\mu_2,\mu_3,r)\} \end{equation} 
where 
\begin{align} g(r) & =-1 + 33 r^2 - 202 r^4 + 146 r^6 - 117 r^8 + 13 r^{10} \label{eq:trap_g}\\
\begin{split} f_1(\mu_2,\mu_3,\mu_4,r) & = 12048 \mu_2 - 2827 \mu_3 + 4352 \mu_4 - 257984 \mu_2 r^2 + 168410 \mu_3 r^2 + 
  182848 \mu_2 r^4 \\ & - 119616 \mu_3 r^4 - 160576 \mu_2 r^6 + 107718 \mu_3 r^6 + 
  18096 \mu_2 r^8 - 12181 \mu_3 r^8
\end{split}\\
f_2(\mu_1,\mu_2,\mu_3,r)&=f_1(\mu_3,\mu_2,\mu_1,r)
\end{align}
First note that the polynomial $f_1(\mu_2,\mu_3,\mu_4,r)$ is linear in $\mu_2$, $\mu_3$, and $\mu_4$. Specifically we can solve  $f_1(\mu_2,\mu_3,\mu_4,r)=0$ for $\mu_4$ in the form $\mu_4=a(r)\mu_2+b(r)\mu_3$, where 
\begin{align} a(r) & =\tfrac1{272}(-753 + 16124 r^2 - 11428 r^4 + 10036 r^6 - 1131 r^8)  \label{eq:a(r)}\\
 b(r) & =\tfrac{1}{4352}(2827 - 168410 r^2 + 119616 r^4 - 107718 r^6 + 12181 r^8)\label{eq:b(r)}
 \end{align}
 Also the third polynomial $f_2(\mu_1,\mu_2,\mu_3,r)$ is linear in $\mu_1$, $\mu_2$, $\mu_3$, and is given by a permutation and substitution variables in $f_1$ so that we can solve for $\mu_1$ in the form $\mu_1=b(r)\mu_2+a(r)\mu_3$. 

Now consider the first polynomial $g(r)$. This is an even function so we make the substitution $r^2=x$ and get the polynomial $f(x) = -1 + 33 x - 202 x^2 + 146 x^3 - 117 x^4 + 13 x^5$, which will have only nonnegative roots. We calculate the Sturm sequence:
\begin{align}
p_0 &= f(x)\\
p_1 & = f'(x)\\
p_{i} & = - \text{Rem}(p_{i-2},p_{i-1}) \text{ for } i = 2,...,5
\end{align}
Let $v(a)$ be the number of sign changes in the Sturm sequence evaluated at $a$. 
We find 
\begin{align}
v(0) & = 4\\
v(1/10) & = 3\\
v(2/10) & = 2 \\
v(7.81) & = 2\\
v(7.82) & = 1
\end{align}
as well as $v(+\infty) = 1$ where the sign of the polynomial at $+\infty$ is the sign of its leading coefficient.
From this we can conclude that $f(x)$ has three positive nonzero real roots in the intervals $(0,1/10]$, $(1/10,2/10]$, and $(7.81,7.82]$. Numerically these are \begin{equation}
x \approx 0.0396673, 0.141048,7.81164. \end{equation}
so that there are six real roots of $g(r)$: 
\begin{equation}
r \approx \pm0.199167,\pm 0.375563,\pm 2.79493 . \label{r2values} \end{equation}
Solving for $\theta_2$ values that correspond to $r$ in the tangent half-angle identities in Equation \eqref{eq:tangent_half_angle}, we get
\begin{align}
 r=\pm 2.79493 &\Leftrightarrow \theta_2 \approx \pm 0.687197\\
r= \pm 0.375563 &\Leftrightarrow \theta_2 \approx \pm 2.42306\\
r= \pm 0.199167 &\Leftrightarrow \theta_2 \approx \pm 2.74840.
\end{align}
We can also give numerical approximations for the corresponding set of circulation parameters using $a(r)$ and $b(r)$ as given in equations \eqref{eq:a(r)} and \eqref{eq:b(r)}:
\begin{align} \begin{split} r=\pm 2.79493 & \Leftrightarrow \theta_2 \approx \pm 0.687197\\
 & \text{ and } (\mu_1,\mu_2,\mu_3,\mu_4) \approx (-0.638032\mu_2 + 1.31061\mu_3,\mu_2,\mu_3,1.31061\mu_2 -0.638032\mu_3)
 \end{split}\\
 \begin{split}
r= \pm 0.375563 &\Leftrightarrow \theta_2 \approx \pm 2.42306 \\
& \text{ and } (\mu_1,\mu_2,\mu_3,\mu_4) \approx (-4.33010\mu_2 + 4.85887\mu_3,\mu_2,\mu_3,4.85887\mu_2 -4.33010\mu_3)
\end{split}\\
\begin{split}
r= \pm 0.199167 &\Leftrightarrow \theta_2 \approx \pm 2.74840 \\
& \text{ and } (\mu_1,\mu_2,\mu_3,\mu_4) \approx (-0.843716\mu_2 -0.480743\mu_3,\mu_2,\mu_3,-0.480743\mu_2-0.843716\mu_3)   
\end{split}
\end{align}
 \hfill $\square$
\end{proof}

The three configurations with $r>0$ from Theorem \ref{thm:3equalsides} are pictured in Figure \ref{three equal sides pictures}. These all satisfy $\theta_2\in(0,\pi)$. The other three configurations are reflections over the horizontal axis. 
Note that two of the possible configurations have $\theta_2>2\pi/3$, and thus the inscribed quadrilateral is not a trapezoid with three equal sides. We get the following corollary.
	
\begin{figure}[h]
\centering
\begin{subfigure}[t]{.3\textwidth}
\includegraphics[width=1.4in]{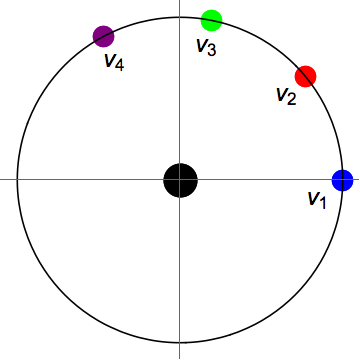}
\caption{$\theta_2 \approx 0.687197$}
\label{threeequal_ex1}
\end{subfigure} 
\begin{subfigure}[t]{.3\textwidth}
\includegraphics[width=1.4in]{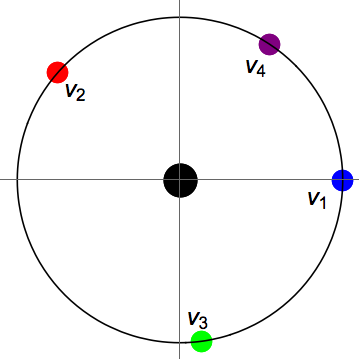}
\caption{$\theta_2 \approx  2.42306$}
\label{threeequal_ex3}
\end{subfigure} 
\begin{subfigure}[t]{.3\textwidth}
\includegraphics[width=1.4in]{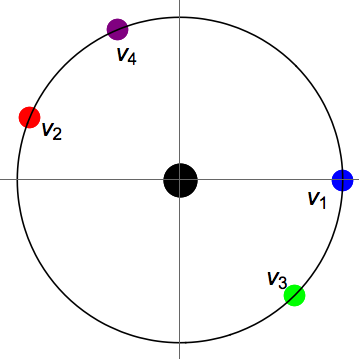}
\caption{$\theta_2 \approx  2.74840$}
\label{threeequal_ex2}
\end{subfigure} 
\caption{The three possible configurations satisfying $\theta_1=0, \theta_3=2\theta_2, \theta_4=3\theta_2$ and $\theta_2 \in (0,\pi)\setminus\{2\pi/3\}$ . The configuration in Figure \ref{threeequal_ex1} is the only configuration with $\theta_2\in (0,2\pi/3)$ and the only true inscribed trapezoid with three equal sides.}
\label{three equal sides pictures}
\end{figure}

\begin{corollary}
There is only one possible inscribed trapezoid with three equal sides in the $(1+4)$-vortex problem with vortex ordering $v_1<v_2<v_3<v_4$. Let $r^*\approx 2.79493$ be the largest positive root of the polynomial $g(r)=-1+33r^2-202r^4+146r^6-117r^8+13r^{10}$. Then the angle of the spacing is given by $\theta^* = \arcsin(2r^*/(1+r^{*2}))\approx 0.687197$ and circulation parameters in the set $\{(a(r^*)\mu_2 + b(r^*)\mu_3, \mu_2,\mu_3,b(r^*)\mu_2 +a(r^*) \mu_3 ) : \mu_2,\mu_3 \in \mathbb{R}\}$, where $a(r)$ and $b(r)$ are given in equations \eqref{eq:a(r)} and \eqref{eq:b(r)}.
\end{corollary}

In Figure \ref{threeequal_ex3} and \ref{threeequal_ex2}, we see examples of symmetric configurations with two degrees of freedom, \textit{general isosceles trapezoids}. These can be characterized by $\theta_i-\theta_j = \theta_k-\theta_{\ell}$ for $i\neq j\neq k\neq \ell$. In these case the angles $(\theta_4-\theta_2)$ and $(\theta_3-\theta_1)$ are equal. The conditions on the circulation parameters $\mu_i$ for these general isosceles trapezoids could be a line of future work.

\subsection{Numerical Results: Linear Stability of the trapezoid with equal sides}

Theorem \ref{thm:3equalsides} proves that there is a critical point of $V(\theta)$ satisfying the symmetry conditions of a trapezoid with three equal sides for all values of the parameters $\mu_2$ and $\mu_3$. Here we examine if any of these critical points are degenerate and the linear stability of the corresponding relative equilibria of the $(1+N)$-body problem.

The properties of nondegeneracy and linear stability depend on the eigenvalues of the Hessian matrix of $V(\theta)$ as described in Definition \ref{def:nondegen} and Theorem \ref{stability theorem}. 
The Hessian of $V(\theta)$, $H = V_{\theta \theta} $ is a symmetric matrix of the form
\begin{align}
H_{ij} & = -\mu_i\mu_j(\cos(\theta_i-\theta_j)-\frac{\cos(\theta_i-\theta_j)}{2-2\cos(\theta_i-\theta_j)}+\frac{2\sin^2(\theta_i-\theta_j)}{(2-2\cos(\theta_i-\theta_j)} \text{for i} \neq j \label{eq:hessian1}\\
H_{ii} & = \sum_{j\neq i} H_{ij}
\label{eq:hessian2}\end{align}
We substitute in the assumptions $\theta_1=0$, $\theta_3=2\theta_2$, $\theta_4=3\theta_2$, $\mu_1=b(r)\mu_2+a(r)\mu_3$, $\mu_4 = a(r)\mu_2+b(r)\mu_3$ into the Hessian, as well as use the command \texttt{TrigExpand}, and the tangent half angle substitutions. Using a working precision of 20, we numerically approximate $r \approx 2.79493$ as the largest root of $g(r)$ in Equation \eqref{eq:trap_g} (corresponding to the only configuration that is a trapezoid with three equal sides) and substitute this value into the Hessian. 

Since degeneracy and linear stability are preserved by multiplication by a positive scalar of the circulation parameters, we restrict to the set $\mu_2^2+\mu_3^2=1$ in the following calculations.

The characteristic polynomial of the Hessian has the form $\lambda^4+a\lambda^3+b\lambda^2+c\lambda$ where the coefficients $a$, $b$, and $c$ are defined in terms of $\mu_2$ and $\mu_3$. We see that $\lambda=0$ is an eigenvalue for $H$ as expected, and we look for $\lambda=0$ to be an eigenvalue with multiplicity two to determine degeneracy.  We solve for $c=0$ and $\mu_2^2+\mu_3^2=1$ where 
\begin{equation}
c=c_1 \mu_2^5 \mu_3 - c_2 \mu_2^4 \mu_3^2 - 
 c_3 \mu_2^3 \mu_3^3 - 
 c_2 \mu_2^2 \mu_3^4 + c_1 \mu_2 \mu_3^5 \end{equation} where $c_1 \approx 53.9615 $, $c_2 \approx 69.9672$, and $c_3 \approx 62.6947$. Using \texttt{Solve} in Mathematica, we find eight degenerate roots of $V(\theta)$, each with a corresponding ray in the $\mu_2\mu_3$-plane. These span the coordinate axes ($\mu_2=0$ and $\mu_3=0$) and two lines $\mu_3=k\mu_2$ where $k \approx 0.486821$ and $k\approx 2.05414$. These are shown as the black lines in Figure \ref{fig:three-linstab}.

To determine linear stability, we numerically calculate eigenvalues of the weighted Hessian $\mu^{-1}H$ for 6284 evenly spaced values on the circle $\mu_2^2+\mu_3^2=1$ (step size of 0.001).  We find that for angles in the interval of $(0.453, 1.118)$ radians, the weighted Hessian has three positive eigenvalues. We also find the region given by reflection through the origin corresponds to parameter values for which the weighted Hessian has three negative eigenvalues. These regions are shown in Figure \ref{fig:three-linstab}.

\begin{figure}[h]
\centering
\includegraphics[height=2.75in]{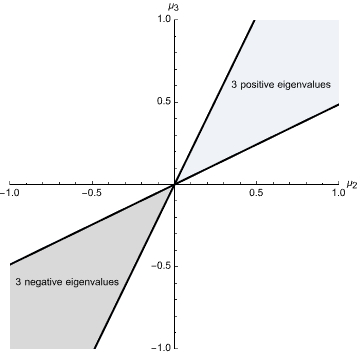}
\caption{Regions corresponding to the linear stability of relative equilibria forming a trapezoid with three equal sides. Critical points for parameter values on the boundaries of the region are degenerate. }
\label{fig:three-linstab}
\end{figure}

\FloatBarrier
 
\section{Kites}
A kite is symmetric by reflection over a line containing two of the infinitesimal vortices. 

\begin{theorem}\label{thm:kite}
Assume $\mu_i\neq 0$. Assume that a configuration of the $(1+4)$-vortex problem is a kite. Then the two vortices that are \emph{not} on the line of symmetry must have equal circulation parameters.\end{theorem}

\begin{proof}Without loss of generality assume $\theta_1=0$ and assume the line of symmetry along the horizontal axis contains $v_1$ and $v_3$ . Thus $\theta_3=\pi$, and configurations of this type can be described as a one-parameter family of configurations $\theta=(0,\theta_2,\pi,-\theta_2)$.

Substituting $\theta_1=0,\, \theta_3=\pi$ and $\theta_4=-\theta_2$ into $\mu^{-1}\nabla V$ and simplifying as described in Section \ref{process} gives the following 
\begin{align}
V_{\theta_2}^{num}&=
-2 ((\mu_1 + \mu_3) \cos(\theta_2) + (\mu_1 - \mu_3) \cos(2 \theta_2) + \mu_4 \cos(3 \theta_2)) \sin^3(
  \theta_2)\\
V_{\theta_3}^{num}&=
   (\mu_2 - \mu_4) (1 + 2 \cos(\theta_2)) \sin(
  \theta_2)\\
V_{\theta_4}^{num}&=
   2  ((\mu_1 + \mu_3) \cos(\theta_2) + (\mu_1 - \mu_3) \cos(2 \theta_2) + 
   \mu_2 \cos(3 \theta_2)) \sin^3(\theta_2)
\end{align}

After removing the $\sin(\theta_2)$ factor from all three equations, expanding trig functions and using the tangent half-angle change of variables, we get the polynomials 
\begin{align}
V_{\theta_2}^{num^*}&=32 (-2 \mu_3 - \mu_4 - 6 \mu_1 r^2 + 4 \mu_3 r^2 + 15 \mu_4 r^2 - 4 \mu_1 r^4 + 
   6 \mu_3 r^4 - 15 \mu_4 r^4 + 2 \mu_1 r^6 + \mu_4 r^6) \label{kites-poly1}\\
V_{\theta_3}^{num^*}&= (\mu_2 - \mu_4) (-1 + 3 r^2) \label{kites-poly2}\\
V_{\theta_4}^{num*}&= 32 (-\mu_2 - 2 \mu_3 - 6 \mu_1 r^2 + 15 \mu_2 r^2 + 4 \mu_3 r^2 - 
   4 \mu_1 r^4 - 15 \mu_2 r^4 + 6 \mu_3 r^4 + 2 \mu_1 r^6 + \mu_2 r^6) \label{kites-poly3}
   \end{align}
   
   We can quickly see that $\mu_2=\mu_4$ and $r=\pm1/\sqrt{3}$ is a root of $V_{\theta_3}^{num^*}$, but it is not apparent that these are roots of the other two equations.
   
Since we are investigating the relation between circulation weights $\mu_i$ that would guarantee symmetry of relative equilibria, we project the ideal $\langle V_{\theta_2}^{num}\textsuperscript{*},V_{\theta_3}^{num}\textsuperscript{*}, V_{\theta_4}^{num}\textsuperscript{*}\rangle$ on to $(\mu_1,\mu_2,\mu_3,\mu_4)$-space. We specify an elimination ordering $r > \mu_1> \mu_2> \mu_3> \mu_4$ and use a Gr\"{o}bner basis to calculate the elimination ideal in \textit{grevlex} ordering that removes configuration variable $r$. We perform Gr\"{o}bner basis algorithm through Mathematica to obtain the following basis for the elimination ideal:
\begin{equation}
\{\mu_2-\mu_4\}
\end{equation}
Thus $\mu_2=\mu_4$. In this proof, we assumed the line of symmetry contained the vortices $v_1$ and $v_3$, thus the vortices not on the line must have equal circulation. The symmetry $V(\theta)$ means the result will be the same for any reordering of the vortices.  \hfill $\square$ \end{proof}

Next we ask the question, what actual angles $\theta_2$ correspond to solutions that are kite configurations, and how are they related to the ratio of the circulation parameters? 

\begin{example}\label{V_theta_3 root} 
In $V_{\theta_3}^{num^*}$, the root $r=\pm 1/\sqrt{3}$, corresponding to $\theta_2=\pm 2\pi/3$ is apparent. When substituted into $\nabla V$, the gradient becomes
$$\nabla V = \frac1{\sqrt{3}}(\mu_1(\mu_4-\mu_2),\mu_2(\mu_1-\mu_4),0,\mu_4(\mu_2-\mu_1))$$
Thus the solutions for $\mu_i\neq 0 $ are $\mu_2=\mu_4=\mu_1$. Checking for nondegeneracy of the critical point, we look at the Hessian $V_{\theta\theta}$. When substituting $\mu_2=\mu_1$, $\mu_4=\mu_1$ and the configuration $\theta_1=0, \theta_2=\pm 2\pi/3$, $\theta_3=\pi$, and $\theta_4=-\theta_2$ into $V_{\theta\theta}$, the eigenvalues are 
\begin{align*}
\lambda_1& = 0, \qquad \lambda_2= -\tfrac12 \mu_1 (\mu_1 - 3 \mu_3),\\
\lambda_{3,4}&= \tfrac14 \mu_1 \left(-\mu_1 + 6 \mu_3 \pm \sqrt{\mu_1^2 + 12 \mu_1 \mu_3 + 108\mu_3^2}\right)
   \end{align*}
Degenerate solutions occur when there is more one zero eigenvalue. We see $\lambda_2=0$ is zero when $\mu_1=3\mu_3$ and one of $\lambda_{3,4}$ vanishes $\mu_1=-3\mu_3$. So the solution is nondegenerate for any values of $\mu_1=\mu_2=\mu_4 \neq \pm 3\mu_3$.  

The linear stability of the relative equilibria for $\varepsilon>$ continued from this configuration is determined by the eigenvalues of the weighted Hessian $\mu^{-1}V_{\theta\theta}$. Using the same substitutions, we get 
\begin{align*}
\lambda_1&=0, \qquad \lambda_2=\tfrac12 (-\mu_1 + 3 \mu_3)\\
\lambda_{3,4}&=\tfrac18 \left(7 \mu_1 + 3 \mu_3 \pm \sqrt{121 \mu_1^2 + 282 \mu_1 \mu_3 + 81 \mu_3^2}\right)
\end{align*}
Using \texttt{Reduce} on the equations $\lambda_2>0$, $\lambda_3>0$, and $\lambda_4>0$ in Mathematica, we find that there are $N-1=3$ positive eigenvalues when $\mu_3>0$ and $\frac{3}{121} (-47 + 4 \sqrt{70}) \mu_3 \approx -0.335544 \mu_3 \leq \mu_1 \leq -\frac13\mu_3$. Thus there is a small range of ratios $\mu_1=\mu_2=\mu_4:\mu_3$ where kites at an angle of $\theta_2=2\pi/3$ are linearly stable. 
\end{example}

\begin{theorem}\label{thm:kite configs}
Assume $\mu_2=\mu_4$. Then there is at least one kite configuration such that $\theta_1=0$, $\theta_3=\pi$ and the vortices are ordered $v_1<v_2<v_3<v_4$ counter clockwise around the circle. Additionally, define the rational expressions:
\begin{align} 
a = &\,\,  \tfrac{1}{(2+\mu_2)}(6\mu_3-15\mu_2-4)\\
b = &\,\, \tfrac{1}{(2+\mu_2)}(4\mu_3+15\mu_2-6)\\
c= &\,\, \tfrac{1}{(2+\mu_2)}(-\mu_2-2\mu_3)\\
\Delta= &\,\, a^2b^2-4b^3-4a^3c+18abc-27c^2 
\end{align}

\noindent Then
\begin{enumerate}
 		\item There are three kite configurations for parameters sets that are scalar multiples of vectors $(1,\mu_2,\mu_3,\mu_2)$ where $\mu_2$ and $\mu_3$ satisfy 
		$\{(\mu_2,\mu_3): c<0, b>0, \Delta>0, ab-9c<0 \}$	
   		\item There are two kite configurations for parameters sets that are scalar multiples $(1,\mu_2,\mu_3,\mu_2)$ where $\mu_2$ and $\mu_3$ are in the union of the sets below:
     \begin{enumerate}
     \item $\{(\mu_2,\mu_3) : c> 0 \}$
     \item $\{(\mu_2,\mu_3) : c =0$ and $\mu_2<-2$ or $\mu_2>6/13 \}$
     \item $\{(\mu_2,\mu_3) : \mu_2 =-2$ and $-\tfrac{13}{3} < \mu_3<1  \}$
     \item $\{\mu_2,\mu_3): \Delta=0$ and $ab-9c < 0 \}$
    \end{enumerate}
\item There is only one kite configuration for any other parameter sets that are scalar multiplies of the vectors $(1,\mu_2,\mu_3,\mu_2)$.	
	\end{enumerate}
\end{theorem}

\begin{figure}[h]
\centering
\includegraphics[height=2.75in]{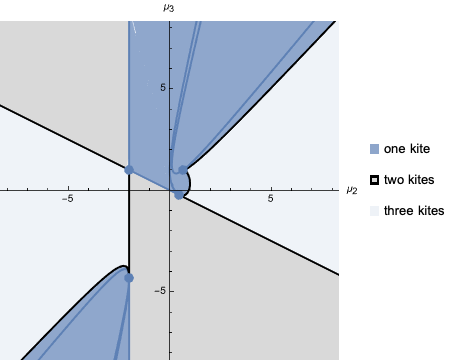}
\caption{The partition of the $\mu_1=1$ plane in $\mu_1\mu_2\mu_3$-parameter space for kite configurations. There is one kite configuration in the blue shaded regions, along the blue curves and at the blue dots. There are two kite configurations in the gray regions and along the gray curves. There are three kite configurations in the rest of the plane.}
\label{fig:kite-roots}
\end{figure}

\begin{figure}
\centering
\includegraphics[width=.9\textwidth]{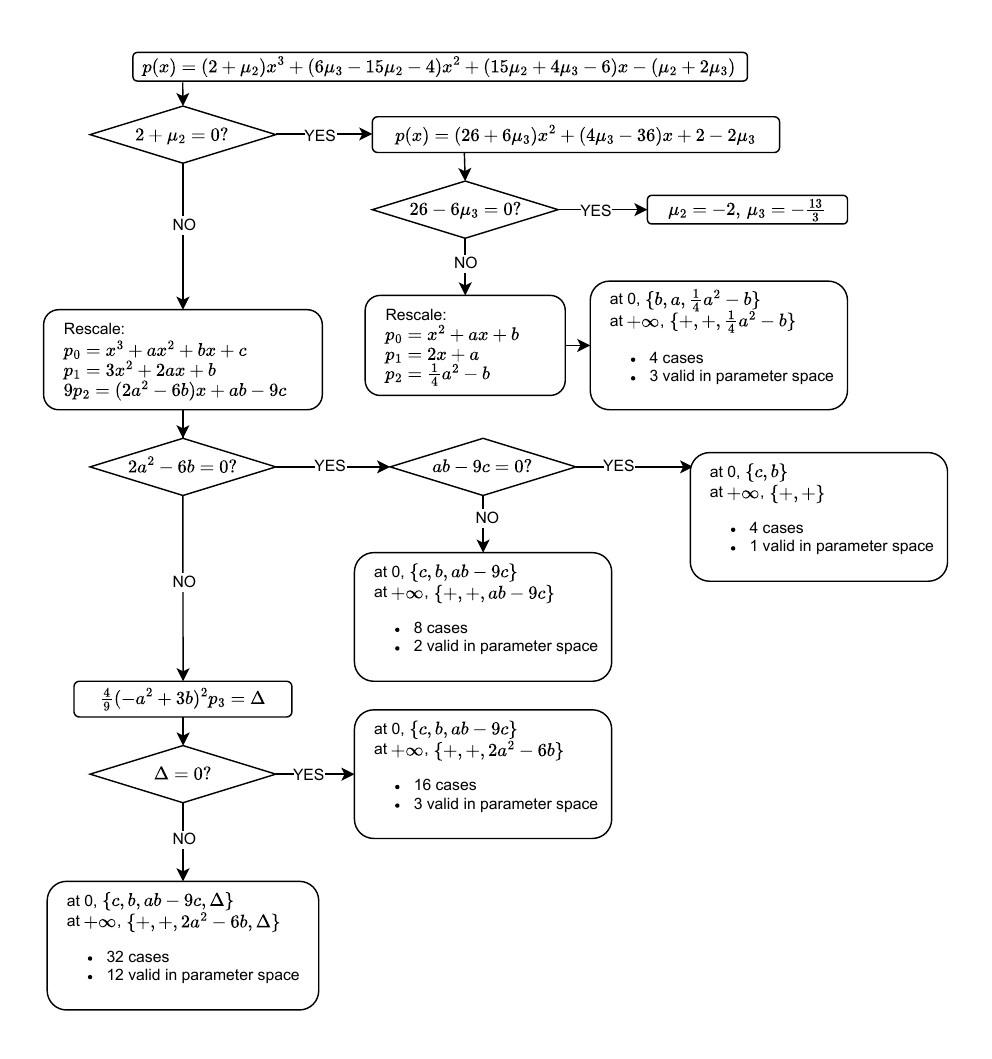}
\caption{The computation tree of the Sturm sequence in the proof of Theorem \ref{thm:kite configs}. }
\label{fig:kite-flowchart}
\end{figure}

\begin{proof}
In $V_{\theta_4}^{num*}$ \eqref{kites-poly3}, we see the factor \begin{equation}p(\mu_1,\mu_2,\mu_3,r)=-\mu_2 - 2 \mu_3 - 6 \mu_1 r^2 + 15 \mu_2 r^2 + 4 \mu_3 r^2 - 4 \mu_1 r^4 - 
 15 \mu_2 r^4 + 6 \mu_3 r^4 + 2 \mu_1 r^6 + \mu_2 r^6 \label{kitesconfig}
 \end{equation}
 This factor is the same the one in $V_{\theta_2}^{num*}$ \eqref{kites-poly1} with the substitution $\mu_4=\mu_2$. Roots of $p$ where $r>0$ correspond to the ordering of vortices $v_1<v_2<v_3<v_4$. 
 Change variables $x=r^2$ in $p$ so that we get the polynomial 
\begin{equation}
p(x) = (2\mu_1+\mu_2)x^3 + (6\mu_3-4\mu_1-15\mu_2)x^2+(4\mu_3+15\mu_2-6\mu_1)x-\mu_2-2\mu_3 \label{eq:original p}
\end{equation}
with three parameters. Since $p(x)$ is an odd polynomial, it has at least one real root, so there is always at least one kite configuration. Also since $p$ is homogeneous $p(k\mu_1,k\mu_2,k\mu_3,x)=kp(\mu_1,\mu_2,\mu_3,x)$, we can choose a scaling with $\mu_1=1$. 

Then we have $p(x)=(2+\mu_2)x^3 + (6\mu_3-4-15\mu_2)x^2+(4\mu_3+15\mu_2-6)x-\mu_2-2\mu_3$.

We proceed to count roots of $p(x)$ in the interval $(0,\infty)$. (We exclude roots at $x=r=0$ where $\theta_2=\pi$ and $v_2$ collides with $v_3$.) We use the algorithm described in section 1.3.2 of \cite{Coste} in order to count the positive roots of $p(x)$ for different cases of the parameters. Each step in the algorithm considers if the leading terms of the polynomials in the Sturm sequence are zero or not. A computation tree for the algorithm is depicted in Figure \ref{fig:kite-flowchart}.

We first consider the case where the leading coefficient of $p(x)$ is zero, i.e. when $\mu_2=-2$. Then $p(x)= (26+6\mu_3) x^2 +(4\mu_3-36)x+ 2 - 2 \mu_3$. 
If the new leading coefficient of $p(x)$ is also zero, i.e. $\mu_3=-\frac{13}3$, then $p(x)=\frac13(32- 160 x)$ and has one root, $x=\frac15$. This is marked by a blue dot at $(-2,-\tfrac{13}{3})$ in Figure \ref{fig:kite-roots}.

If $\mu_3\neq -\frac{13}3$, then we rescale $p(x)$ to get $p_0(x)=x^2+ax+b$ where $a=\frac{(4\mu_3-36)}{(26+6\mu_3)}$ and $b=\frac{2-2\mu_3}{(26+6\mu_3)}$. We create a Sturm sequence 
\begin{align}
	p_0(x) & =x^2+ax+b\\
	p_1(x) & = p_0'(x) = 2x+a\\
	p_2(x) & = -\text{Rem}(p_0,p_1) = \frac14 a^2 -b
\end{align}
Note that $p_2(x)=\frac14a^2-b = \frac{4 (17 - 2 \mu_3 + \mu_3^2)}{(13 + 3 \mu_3)^2}$. Since the discriminant of $17-2\mu_3+\mu_3^2$ is negative, and so $p_2(x)$ is always positive. At $x=0$, the Sturm sequence is $\{b,a,+\}$ and at $x=+\infty$, the Sturm sequence is $\{+,+,+\}$ with $v(+\infty)=0$. We get four cases based on the signs of $b$ and $a$. Using the \texttt{Reduce} command in Mathematica, we find only three of these cases are valid in parameter space. These cases are summarized in Table \ref{table:strum mu2=-2}. Along with the point $(-2,-\tfrac{13}{3})$, these cases completely partition the line $\mu_2=-2$ as shown in Figure \ref{fig:kite-roots}.

\begin{table}[h]
\centering
\hspace*{-.25in}
\caption{Summary of cases in the Sturm sequence when $\mu_2=-2$ and $\mu_3 \neq -\tfrac{13}{3}$. \vspace*{11pt}}
    \begin{tabular}{p{.65in}|p{.8in}|p{.9in}|c|p{2in}}
     Assumptions & \multicolumn{2}{c|}{Sturm Sequence} & $v(0)-v(+\infty)$  & Description  \\ \cline{2-3}
     & at $x=0$ & at $x=+\infty$ & &   \\ \hline
    \hline
    $b < 0$, $a > 0$ & $\{b,a,\tfrac14a^2-b\}$ \newline $=\{-,+,+\}$  & $\{+,+,\tfrac14a^2-b\}$ \newline $=\{+,+,+\}$ & $1 -0 = 1$ & $\mu_3< -\tfrac{13}{3} $ and $ \mu_3 \geq 9$\\ \hline
    $b\leq 0$, $a \leq 0$ & $\{b,a,\tfrac14a^2-b\}$ \newline $=\{-,-,+\}$\newline or $\{0,-,+\}$ \newline or $\{-,0,+\}$ & $\{+,+,\tfrac14a^2-b\}$ \newline $=\{+,+,+\}$ & $1 -0 = 1$ & $1 \leq \mu_3 \leq 9$\\ \hline
    $b>0$, $a<0$ & $\{b,a,\tfrac14a^2-b\}$ \newline $=\{+,-,+\}$ & $\{+,+,\tfrac14a^2-b\}$\newline $=\{+,+,+\}$ & $2 -0 = 2$ & $-\tfrac{13}{3} < \mu_3 < 1$\\ \hline
    $b>0$, $a>0$ & $\{b,a,\tfrac14a^2-b\}$ \newline $=\{+,+,+\}$ & $\{+,+,\tfrac14a^2-b\}$ \newline $=\{+,+,+\}$ & $0-0=0$ & Not valid in parameter space \\ \hline 
    \end{tabular}
\label{table:strum mu2=-2}
\end{table}

Next we consider the case where $\mu_2\neq -2$. We rescale $p(x)$ in \eqref{eq:original p} to have a leading coefficient of 1 and define the following rational expressions in the parameters $\mu_2,\mu_3$:
\begin{align}
a & =  \tfrac{1}{(2+\mu_2)}(6\mu_3-15\mu_2-4)\\
b & = \tfrac{1}{(2+\mu_2)}(4\mu_3+15\mu_2-6)\\
c & = \tfrac{1}{(2+\mu_2)}(-\mu_2-2\mu_3)
\end{align}
Then the Sturm sequence is
\begin{align}
p_0(x)&=x^3+ax^2+bx+c\\
p_1(x)& = 3x^2+2ax+b\\
9p_2(x)&=9(-\text{Rem}(p_0,p_1))=(2a^2-6b)x+ab-9c\\
p_3(x)& = -\text{Rem}(p_1,p_2) =\frac{\Delta}{\frac{4}{9}(-a^2+3b^2)^2}=\frac{a^2b^2-4b^3-4b^3-4a^3c+18abc-27c^2}{\frac{4}{9}(-a^2+3b^2)^2}
\end{align} 

Since the signs of $9p_2$ and of $\Delta$ are the same as $p_2$ and $p_3(x)$, respectively, so we use $9p_2$ and $\Delta$ in our Sturm sequences for simplicity. We use a similar process as above of considering when each new leading coefficient is zero or nonzero, as shown in the computation tree in Figure \ref{fig:kite-flowchart}. 

Since the leading terms of $p_0$ and $p_1$ are constant, we consider if the leading term of $9p_2$ is zero or not. We start by assuming $2a^2-6b=0$ and $ab-9c=0$, so that the Sturm sequence $\{p_0,p_1\}$. Evaluated at $x=0$, the sequence is $\{c,b\}$, and at $x=+\infty$, the sequence is $\{+,+\}$. There are four possible subcases for the Sturm sequence at $x=0$: $\{+,+\}$, $\{+,-\}$, $\{-,+\}$ and $\{+,+\}$. Using the \texttt{Reduce} command in Mathematica, the case where $2a^2-6b=0$, $ab-9c=0$, $c<0$ and $b>0$ was found to be the only case that is valid in parameter space. In fact the hyperbola $2a^2-6b=0$ is contained in the region where $c<0$ and $b>0$. Its intersection with the hyperbola $ab-9c=0$ is one point: $\mu_2=\tfrac23$, $\mu_3=1$. This is case 1 in Table \ref{table:sturm first part}.

We next consider the case where $2a^2-6b=0$ and $ab-9c\neq 0$, so that $9p_2 = ab-9c$. Then the Sturm sequence is $\{p_0,p_1,ab-9c\}$. Evaluated at $x=0$, this is $\{c,b,ab-9c\}$ and at $x=+\infty$, this is $\{+,+,ab-9c\}$. There are 8 possible cases for the Sturm sequences. Using \texttt{Reduce} in Mathematica, only two are found to be valid in parameter space. These are summarized in rows 2 and 3 of Table \ref{table:sturm first part}.

We continue this process assuming $2a^2-6b \neq 0$ and calculate the Sturm sequence when $\Delta =0$ and when $\Delta \neq 0$ as shown in the flow chart in Figure \ref{fig:kite-flowchart}. Cases that are valid in parameter space are summarized in Table \ref{table:sturm first part} for $\Delta =0$ and in Table  \ref{table:sturm cases} for $\Delta \neq 0$.

Overall, we find one case where there are three positive roots of $p(x)$, giving three unique kite configurations. This region is defined by $c<0$, $b>0$, $ab-9c<0$, $\Delta>0$. In parameter space it is in two disjoint regions (see Figure \ref{fig:kite-roots}), bounded by $\mu_3 = -\tfrac12\mu_2$, $\mu_2=-2$ and $\Delta =0$ where $ab-9c < 0$, but not including the boundary. There are two positive roots of $p(x)$, giving two unique kite configurations, along the boundary of the three kite region, as well as in the region defined by $c>0$. In parameter space this is two disjoint regions where $\mu_3\geq -\tfrac12\mu_2$ for $\mu_2<-2$ and $\mu_3<-\tfrac12\mu_2$ for $\mu_2>-2$. In the rest of the plane, there is only one positive root of $p(x)$, giving one unique kite configuration. This includes entirety of the hyperbola given by $2a^2-6b=0$, the curve $\Delta=0$ when $ab-9c>0$, the line $\mu_2=-2$ for $\mu_3 \geq 1$ and for $\mu_3 \leq -\tfrac{13}{3}$, and when $\mu_3=-\tfrac12\mu_2$ for $-2\leq \mu_2 \leq \tfrac{6}{13}$.

 This completely partitions the $\mu_1=1$ plane in $\mu_1\mu_2\mu_3$-space, as shown in Figure \ref{fig:kite-roots}.  
\hfill $\square$
\end{proof}

\begin{table}[h]
\centering
\hspace*{-.25in}
\caption{Sturm cases when $\mu_2\neq -2$ and along the curves $2a^2-6b=0$ and $\Delta=0$. The sections of the table correspond to yes or no answers in the flow chart in Figure \ref{fig:kite-flowchart}.  Cases for signs in the Sturm sequence that are not valid in parameter space are omitted from the table. \vspace*{11pt}}
\begin{tabular}{l|p{.65in}|p{.8in}|p{.9in}|c|p{2in}}
 & Assumptions & \multicolumn{2}{c|}{Sturm Sequence} & $v(0)-v(+\infty)$  & Description  \\ \cline{3-4}
& & at $x=0$ & at $x=+\infty$ & &   \\ \hline
\multicolumn{6}{c}{Case where $2a^2-6b=0$ and $ab-9c=0$} \\
\hline
1& $2a^2-6b=0$
\newline $ab-9c=0$
& $\{c,b\}=\{-,+\}$  & $\{+,+\}$& 1 & The hyperbola $2a^2-6b=0$ is contained in the region where $c<0 $ and $b>0$. The intersection of $2a^2-6b=0$ and $ab-9c=0$ is one point where $\mu_2=\tfrac23$ and $\mu_3=1$. The curve $\Delta=0$ also contains this point.
  \\
 \hline \multicolumn{6}{c}{Cases where $2a^2-6b=0$ and $ab-9c \neq 0$}\\\hline
2 & $2a^2-6b=0$ \newline
$ab-9c>0$ \newline 
& $\{c,b,ab-9c\} = \{-,+,+\}$ & $\{+,+,+\}$ \newline \vspace*{.75in} & 1
& \multirow{2}{2in}{The hyperbola $2a^2-6b=0$ is contained in the region where $c<0 $ and $b>0$. Valid subcases divide $2a^2-6b=0$ into regions where $ab-9c$ is positive or negative. 
\includegraphics[width=2in]{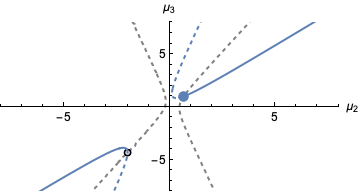} 
\includegraphics[width=1.5in]{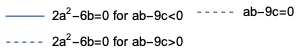}}  \\\cline{1-5} 
3 & $2a^2-6b=0$ \newline
$ab-9c<0$ \newline & $\{c,b,ab-9c\} = \{-,+,-\}$ & $\{+,+,-\}$ \newline \vspace*{.75in} & 1 &  \\ \hline 
\multicolumn{6}{c}{Cases where $\Delta=0$ and $2a^2-6b\neq 0$} \\ \hline
4 & $\Delta =0$ \newline $c<0$,
$b>0$, \newline $ab-9c<0$
& $\{c,b,ab-9c\} =$ \newline$ \{-,+,-\}$ & $\{+,+,2a^2-6b\}$\newline $=\{+,+,+\}$ \newline \vspace*{.4in} & 2 & 
\multirow{3}{2in}{The curve $\Delta=0$ is defined in the region where $c<0$ and where $2a^2-6b>0$. Subcases partition the curve $\Delta =0$ where $ab-9c$ is positive or negative or zero. The point at the intersection of the $\Delta =0$ and $ab-9c=0$ is included in case 5. 
\includegraphics[width=2in]{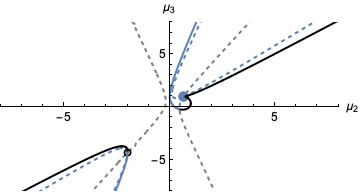} 
\includegraphics[width=1.5in]{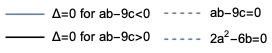}}   \\ \cline{1-5}
5 & $\Delta =0$ \newline $c<0$,
$b>0$, \newline$ab-9c\geq 0$ & $\{c,b,ab-9c\}$ \newline$ = \{-,+,+\}$ \newline or $\{-,+,0\}$ &$\{+,+,2a^2-6b\}$\newline $=\{+,+,+\}$ \newline \vspace*{.4in} & 1 &  \\ \cline{1-5} 
6 & $\Delta =0$ \newline $c<0$, 
$b<0$, \newline$ab-9c>0$ & $\{c,b,ab-9c\} = \{-,-,+\}$ & $\{+,+,2a^2-6b\}$\newline $=\{+,+,+\}$\newline \vspace*{.35in} & 1 &  \\ \hline
\end{tabular}
\label{table:sturm first part}
\end{table}

\begin{table}[h]
\centering
\caption{Summary of cases in the Sturm sequence when $\mu_2\neq-2$ and $\Delta\neq 0$.  Cases for signs in the Sturm sequence that are not valid in parameter space are omitted from the table. }
\label{table:sturm cases}
\begin{tabular}{l|p{.65in}|p{.8in}|p{.9in}|c|p{2in}}
 & Assumptions & \multicolumn{2}{c|}{Sturm Sequence} & $v(0)-v(+\infty)$  & Description  \\ \cline{3-4}
& & at $x=0$ & at $x=+\infty$ & &   \\ \hline
\multicolumn{6}{c}{Cases where $\Delta \neq 0$ and $2a^2-6b\neq 0$} \\\hline
8 & $c<0$, $b>0$, $ab-9c<0$, $\Delta>0$, \newline $2a^2-6b>0$ & $\{c,b,ab-9c,\Delta\} $ $ =\{-,+,-,+\}$ & $\{+,+,2a^2\!-\!6b,\Delta\}$ $=\{+,+,+,+\}$ & 3 & These assumptions result in the maximum number of sign changes for $v(0)$ and the minimum number for $v(+\infty)$ 
\includegraphics[width=2in]{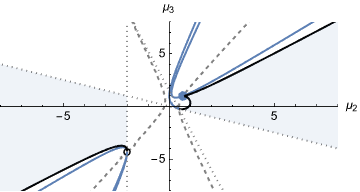} 
\includegraphics[width=1.5in]{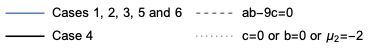}\\ \hline \hline
9 & $c\geq 0$, $b\geq0$, $ab-9c<0$ & $\{c,b,ab\!-\!9c,\Delta\}$ \newline $= \{+,+,-,+\}$ or $\{+,0,-,+\}$ or $\{0,+,-,+\}$ \newline \vspace*{.1in}& $\{+,+,2a^2\!-\!6b,\Delta\}$ $=\{+,+,+,+\}$  & 2 & \multirow{3}{2in}{When $c\geq 0$, $2a^2-6b>0$ and $\Delta\geq 0$. Subcases divide this region where $b$ and $ab-9c$ are positive, negative, or zero.
\includegraphics[width=2in]{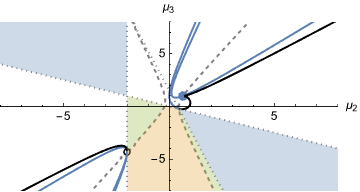}
\includegraphics[width=1in]{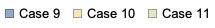}}  \\ \cline{1-5}
10 & $c>0$, $b<0$, $ab-9c\geq 0$ & $\{+,-,+,+\}$ or $\{+,-,0,+\}$ & $\{+,+,+,+\}$ \newline \vspace*{.25in} & 2 &  \\ \cline{1-5} 
11 & $c>0$, $b<0$, $ab-9c\leq0$ & $\{+,-,-,+\}$ or $\{+,-,0,+\}$ & $\{+,+,+,+\}$ \newline \vspace*{.25in} & 2 &  \\ \hline \hline
12 
& $c<0, b\geq 0$,  $\Delta<0$, \newline 
$ 2a^2-6b >0$,  $ab-9c \geq 0$& $\{c,b, ab-9c, \Delta\}$ $=\{-,+,+,-\}$ or $\{-,0,+,-\}$ or $\{-,+,0,-\}$ 
& $\{+,+,2a^2\!-\!6b,\Delta\}$ $=\{+,+,+,-\}$ & 1 & \multirow{5}{2in}{Partitions of the region where $\Delta <0$. Cases contain their boundaries where $b=0$ or $ab-9c=0$. \includegraphics[width=2in]{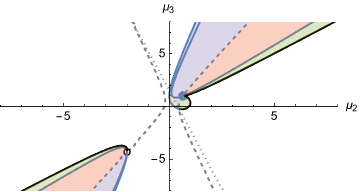}
\includegraphics[width=1.4in]{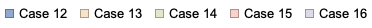} } \\ \cline{1-5}
13 
& $c<0$, $b\leq0$, $\Delta >0$, \newline $2a^2-6b>0$, $ab-9c \geq 0$ & $\{c,b, ab-9c, \Delta\}$ $=\{-,-,+,-\}$ or $\{0,-,+,-\}$ or $\{-,0,+,-\}$ & $\{+,+,2a^2\!-\!6b,\Delta\}$ $=\{+,+,+,-\}$ & 1 & 
\\ \cline{1-5}
14 
& $c<0$, $b>0$, $\Delta<0$ \newline $2a^2-6b>0$, $ab-9c \leq 0$  & $\{c,b, ab-9c, \Delta\}$ $=\{-,+,-,-\}$ or $\{-,+,0,-\}$ & $\{+,+,2a^2\!-\!6b,\Delta\}$ $=\{+,+,+,-\}$ & 1 & \\ \cline{1-5}
15 
&$c<0$, $b>0$, $\Delta<0$ \newline $2a^2-6b<0$, $ab-9c \leq 0$ & $\{c,b, ab-9c, \Delta\}$ $=\{-,+,-,-\}$ or $\{-,+,0,-\}$ & $\{+,+,2a^2\!-\!6b,\Delta\}$ $=\{+,+,-,-\}$ & $1$  & \\ \cline{1-5} 
16 
& $c<0$, $b>0$, \newline$\Delta<0$,\newline  $2a^2-6b<0$,\newline $ab-9c\geq0$ & $\{c,b, ab-9c, \Delta\}$ $=\{-,+,+,-\}$ or
$\{-,+,0,-\}$ & $\{+,+,2a^2\!-\!6b,\Delta\}$ $=\{ +,+,-,-\}$ & 1 &  \\ \hline
17 
& $c\leq 0$, $b\leq0$, \newline$\Delta>0$,\newline  $2a^2-6b>0$,\newline $ab-9c\geq0$& $\{c,b, ab-9c, \Delta\}$ $=\{-,-,+,+\}$ or $\{-,0,+,+\}$ or $\{0,-,+,+\}$ \vspace*{.1in} & $\{+,+,2a^2\!-\!6b,\Delta\}$ $=\{+,+,+,+\}$  
 & 1 & \multirow{5}{2in}{Three regions where $\Delta >0$ and $c\leq 0$. Cases contain their boundaries where $c=0$, $b=0$, or $ab-9c=0$.
\includegraphics[width=2in]{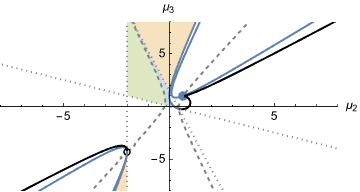}
\includegraphics[width=1.1in]{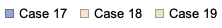} } \\ \cline{1-5}
18 
& $c<0, b\geq 0$,  $\Delta>0$, \newline $ 2a^2-6b >0$, \newline $ab-9c > 0$ & $\{c,b, ab-9c, \Delta\}$  $=\{-,+,+,+\}$  or $\{-,0,+,+\}$ 
& $\{+,+,2a^2\!-\!6b,\Delta\}$ $=\{+,+,+,+\}$ & 1 & \\ \cline{1-5}
19 
& $c\leq 0, b < 0$,  $\Delta>0$, \newline $ 2a^2-6b >0$, \newline $ab-9c \leq 0$ & $\{c,b, ab-9c, \Delta\}$ $=\{ -,-,-,+\}$ or $\{0,-,-,+\}$ &$\{+,+,2a^2\!-\!6b,\Delta\}$ $=\{+,+,+,+\}$ & 1 & \\ \hline
\end{tabular}
\end{table}

\begin{corollary}
 \label{thm: kites-squares}
When parameters are in two equal pairs, one kite configuration is a square.  
 \end{corollary}

The proof of the corollary is simple: $x=1$, corresponding to $\theta_2=\tfrac{\pi}{2}$ is a root of $p(x)$ in Equation \eqref{eq:original p},  if and only if $\mu_1=\mu_3$. From this we see that there are symmetric kite configurations with two equal pairs of circulation parameters that are not squares, for example, where the line $\mu_3=1$ intersects regions where there are two or three kite configurations in Figure \ref{fig:kite-roots}.

\FloatBarrier
\subsection{Numerical Results: Linear Stability of Kite Configurations}

We now determine which parameter values give degenerate kite-shaped critical points of $V(\theta)$ and the linear stability of the corresponding relative equilibria. In $\nabla V(\theta)$ and the Hessian of $V(\theta)$, we make the following substitutions: $\theta_1=0$, $\theta_3=\pi$, $\theta_4=-\theta_2$, and $\mu_4=\mu_2$. For consistency with the previous section we set $\mu_1=1$.

The characteristic polynomial of the Hessian is a quartic polynomial of the form $\lambda^4+a\lambda^3+b\lambda+c$ where the coefficients $a$, $b$, and $c$ are defined in terms of $\mu_2$, $\mu_3$ and $\theta_2$. For two (or more) zero eigenvalues, $c$ must also be zero. We use the \texttt{Coefficient} command in Mathematica to define $c$ in terms of $\mu_2$, $\mu_3$, and $\theta_2$, and then use \texttt{Solve} on the system of equations $\nabla V=0$ and $c=0$. This gives solutions for $\mu_2$ and $\mu_3$ parameterized by $\theta_2$. We find degenerate critical points when $\mu_2=0$ and $\mu_3=0$ and along the orange curves in Figure \ref{fig:kites-degen}.

\begin{figure}[h]
\centering
\includegraphics[height=2.75in]{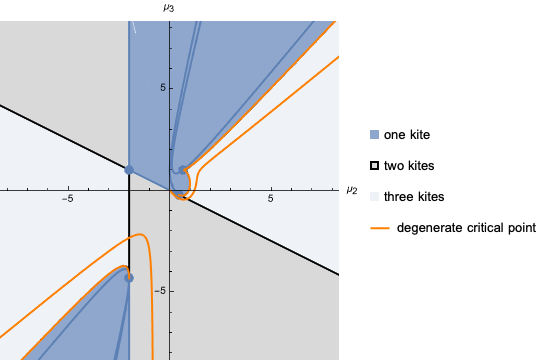}
\caption{Parameter values in the $\mu_1=1$ subset that have a degenerate kite critical point. Degeneracy is preserved by scaling.}
\label{fig:kites-degen}
\end{figure}

Next we consider which kite critical points of $V$ correspond linearly stable relative equilibria. With $\mu_1=1$, we need to examine the $\mu_2\mu_3$-plane, but choose to focus on the square $[-8,8]\times [-8,8]$. This region captures smaller structures near the origin, while showing the larger regions we expect to continue in the rest of the plan.

We sample uniformly at random 15,000 points from the set $[-8,8]\times [-8,8]$. At each point, we numerically solve for the kite-shaped critical points of $V$. We count the number of unique roots with $0<\theta_2<\pi$. The results of this calculation are plotted in Figure \ref{fig:kites-roots-num} and confirm the algebraically determined regions from Theorem \ref{thm:kite configs} as shown in Figure \ref{fig:kite-roots}.

\begin{figure}[h]
\centering
\includegraphics[height=2.75in]{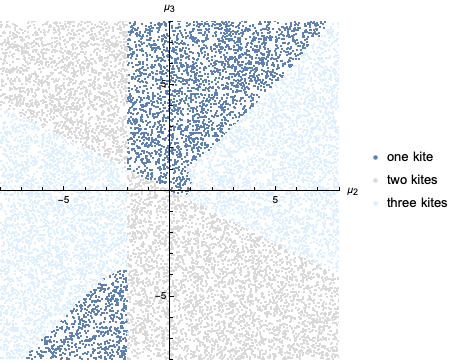}
\caption{Number of kite configurations calculated numerically counts for 15,000 points sampled uniformly at random from the set $[-8,8]\times[-8,8]$. Compare to Figure \ref{fig:kite-roots}.}
\label{fig:kites-roots-num}
\end{figure}

We then calculate the eigenvalues of the weighted Hessian for each root at each point sampled and determine if any have three positive eigenvalues or three negative eigenvalues. We sampled an additional 5,000 points in the set $[-1,-0.5]\times[-8,-2]$ to get more details of the structure. These results are plotted in Figure \ref{fig:kites-linstab}. Parameter sets for critical points with three positive eigenvalues correspond to linearly stable relative equilibria. Multiplication of the parameter set by a positive scalar preserves the stability of the critical point. Parameter sets for critical points with three negative eigenvalues correspond to linearly stable relative equilibria when the parameter set is multiplied by a negative scalar. 

Finally, we consider the angle $\theta_2$ corresponding to critical points of $V$ where the weighted Hessian has either three positive or three negative eigenvalues. These are shown in Figure \ref{fig:kites-angles}.

\begin{figure}[h]
\centering
\includegraphics[height=2.75in]{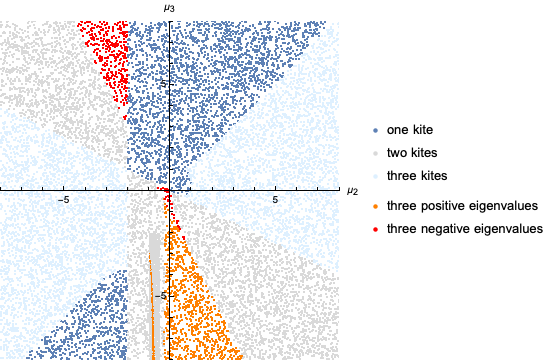}
\caption{Parameter values where the weighted Hessian evaluated at a kite configuration has either three positive eigenvalues or three negative eigenvalues for 15,000 points sampled uniformly at random from the set $[-8,8]\times[-8,8]$ and an additional 5,000 points sampled uniformly at random from the set $[-1,-0.5]\times[-8,-2]$.}
\label{fig:kites-linstab}
\end{figure}
\begin{figure}[h]
\centering
\includegraphics[height=2in]{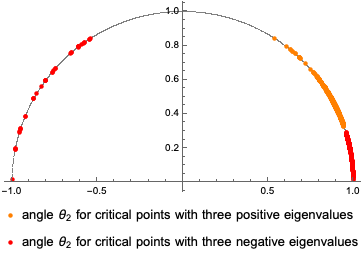}
\caption{Angles for $\theta_2$ associated with critical points of $V$ for kite configurations where the weighted Hessian has either three positive or three negative eigenvalues.}
\label{fig:kites-angles}
\end{figure}

\FloatBarrier

\section{Conclusion}

We found necessary conditions for ratios of circulation parameters for specific types of symmetric configurations in the $(1+N)$-vortex problem with $N=4$. Of the most interest is that the condition of two equal pairs of vortices is necessary for the the $(1+4)$-gon configuration, though it is not sufficient as there are also kite configurations with two equal pairs of vortices. Additionally, there exists only one possible trapezoid with three equal sides, and it does not exist for a set of equal circulation parameters. This shows that the configuration given in Figure 4c of \cite{barry2012relative} does not have equally spaces vortices. Assuming the vortices are ordered $v_1<v_2<v_3<v_4$ counter clockwise around the circle, they inscribe an isosceles trapezoid with interior angles $\theta_4-\theta_3 = \theta_2-\theta_1 \approx 0.66094$ ($37.8691^{\circ}$) and $\theta_3-\theta_2 \approx 0.58762$ ($33.6681^{\circ}$). It is interesting to compare this result with those from the $(1+4)$-body problem where the isosceles trapezoid with four identical infinitesimal masses is estimated to have interior angles of $41.5^{\circ}$ and $37.4^{\circ}$ \cite{albouy2009relative}. Furthermore, the configurations found in Theorem \ref{thm:3equalsides} and pictured in Figures \ref{threeequal_ex2} and \ref{threeequal_ex3} are examples of isosceles trapezoids for unequal circulation parameters, a type of symmetry that could be a point of future inquiry.


\subsection*{CRediT authorship contribution statement}
\textbf{Sophie Le:} Conceptualization, Methodology. 
\textbf{Alanna Hoyer-Leitzel:} Conceptualization, Methodology, Validation, Writing - Original Draft, Writing - Review \& Editing.

\subsection*{Acknowledgements} 
Our deepest appreciation to many anonymous reviewers for helpful feedback, particularly insight into which monomial ordering to use for the Gr\"{o}bner basis calculation in the proof of Theorem \ref{thm:3equalsides}.
A.H-L. thanks Sarah Iams, Heidi Goodson, Annie Raymond, Jessica Sidman and Ashley Wheeler for moments of advice and insight as this paper was developed. A.H-L. is supported by the Kennedy-Schuklenoff endowed chair at Mount Holyoke College, and a portion of this work was completed during A.H-L.'s sabbatical which was partially supported by the Hutchcroft Fund and the Mathematics and Statistics Department at Mount Holyoke College. S.L. completed this work while at student at Mount Holyoke College and received a software fellowship from the Hutchcroft Fund.

\newpage
\bibliography{vortex-bib}{}
\bibliographystyle{siam}


\end{document}